\documentclass{amsart}

\usepackage[english]{babel}
\usepackage[utf8]{inputenc}
\usepackage[T1]{fontenc}

\usepackage[a4paper,top=3cm,bottom=2cm,left=3cm,right=3cm,marginparwidth=1.75cm]{geometry}

\usepackage{amsxtra}
\usepackage{amscd}

    \usepackage{amsthm}
    \usepackage{amssymb}
    \usepackage{amsmath}
    \usepackage{mathtools}
    \usepackage{epsfig}
    \usepackage{caption}
    \usepackage[colorinlistoftodos]{todonotes}
    \usepackage[colorlinks=true, allcolors=blue]{hyperref}
    \usepackage{xcolor}
    \usepackage{pinlabel}
    \usepackage{algorithm}
    \usepackage{algpseudocode}
    \usepackage{relsize}
    \usepackage{geometry}
    \usepackage{subcaption}
    \usepackage[numbers]{natbib}
    \usepackage{hyperref}
    \usepackage{longtable}
    \usepackage[shortlabels]{enumitem}

\theoremstyle{definition}
\newtheorem{definition}{Definition}[section]

\theoremstyle{plain}
\newtheorem{theorem}[definition]{Theorem}
\newtheorem{lemma}[definition]{Lemma}

\newtheorem{remark}[definition]{Remark}
\newtheorem{note}[definition]{Note}
\newtheorem{example}[definition]{Example}

\theoremstyle{remark}

\begin{document}

\title[From plat closure to standard closure of braids and vice-versa]{Passing from plat closure to standard closure of braids in $\mathbb{R}^3$, in handlebodies and in thickened surfaces}

\author[P. Cavicchioli]{Paolo Cavicchioli}
\address{Dipartimento di Matematica,
Università di Bologna, Piazza di Porta S. Donato, 5, 40126, Bologna, Italy. }
\email{paolo.cavicchioli@unibo.it}
\urladdr{https://www.unibo.it/sitoweb/paolo.cavicchioli}

\author[S. Lambropoulou]{Sofia Lambropoulou}
\address{School of Applied Mathematical and Physical Sciences, 
National Technical University of Athens, 
Zografou Campus, 9 Iroon Polytechneiou st., 15772 Athens, Greece.}
\email{sofia@math.ntua.gr}
\urladdr{http://www.math.ntua.gr/~sofia}

\thanks{This research received support from the National Technical University of Athens, the University of Modena and Reggio Emilia and the University of Parma, as well as an STMS travel grant from the COST Action CA17139 - EUTOPIA, all of which we acknowledge thankfully.}

\keywords{standard closure of braids, plat closure of braids, classical knots and links, knots and links in the handlebody, knots and links in thickened surfaces, algorithms and complexity.}

\subjclass[2010]{57K10, 57K12, 57-08, 57-04}



\begin{abstract}
Given a knot or link in the form of  plat closure of a braid, we describe an algorithm to obtain a braid representing the same knot or link with the standard closure, and vice-versa. We analyze the three cases of knots and links: in \(\mathbb{R}^3\), in handlebodies and in thickened surfaces. We show that the algorithm is quadratic in the number of crossings and loop generators of the braid when passing from plat to standard closure, while it is linear when passing from standard to plat closure. 
\end{abstract}

\maketitle

\section*{Introduction} 

By classical results of Brunn and Alexander  \cite{brunn1892uber,alexander1923lemma}, any oriented knot or link in \(\mathbb{R}^3\) is isotopy equivalent to the standard closure of a braid in some braid group \(\mathcal{B}_{n}\). Furthermore, by a theorem of Markov \cite{markov1935freie}, conjugation in the braid groups and the stabilization move generate an equivalence relation in the set of all braids, such that equivalence classes of braids are in bijective correspondence with isotopy classes of oriented knots and links (from now on links). 

The Alexander and Markov theorems were used for the first time in 1984 by V.F.R.~Jones for constructing topological invariants for links described through standard closures of braids, via the surjections of braid group algebras to Templerley-Lieb and Hecke algebras and via linear traces on these algebras supporting the Markov property (see \cite{jones1987hecke} and references therein). 

The Jones construction has been since successfully adapted to other knot algebras and other algebraic/topological settings (see, for example, \cite{turaev1988yang} for the case of Yang-Baxter operators, \cite{wenzl1993braids} for invariants of 3-manifold via local traces, \cite{lambropoulou1994solid, lambropoulou1999knot} for the solid torus using the B-type Hecke algebras, or \cite{kauffman1998centrality} for the case of Hopf algebras, or again \cite{chlouveraki2020identifying, goundaroulis2017framization, flores2018framization} for invariants of framed and classical links from the Yokonuma-Hecke algebras and their Templerley-Lieb-type quotients). 
Moreover, by considering mixed link diagrams in  the 3-sphere, we have the extension of the Alexander and Markov theorems  for links in closed, connected, orientable (from now on c.c.o.) 3-manifolds described via integral or rational topological surgery \cite{lambropoulou1997markov, lambropoulou2006algebraic, diamantis2015braid}, as well as for links in a handlebody \cite{haring2002knot}. 

\begin{figure}[h!]
    \centering
    \includegraphics[width = .85\textwidth]{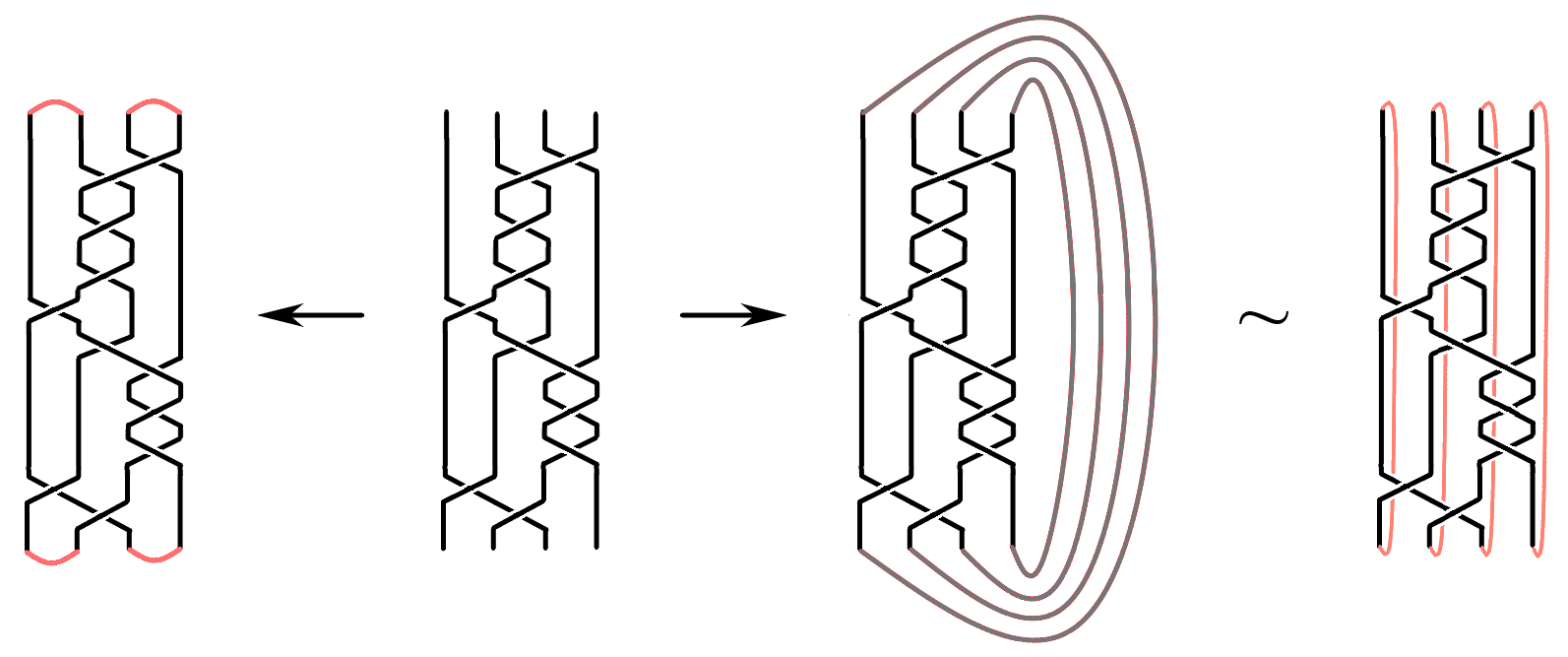}
    \caption{A classical braid closed in the plat  closure and in the  standard closure (with its isotopy to sliding the closing arcs to the back of the braid).}
    \label{fig:both_closure}
\end{figure}

The representation of links as plat closures of braids with an even number of strands (see Figure~\ref{fig:both_closure}) goes back to the works of Hilden \cite{hilden1975generators} and Birman \cite{birman1976stable}, who proved that any link in $\mathbb{R}^3$ may be represented as the plat closure of a braid. For example, all rational knots and links can be represented by a braid with 4 strands, closed via plat closure, leaving the last strand as the identity strand \cite{bankwitz1934viergeflechte}. Hilden \cite{hilden1975generators} and Birman \cite{birman1976stable} also  described a set of moves that connect braids having isotopic plat closures, a result analogous to the classical Markov equivalence. 

The plat closure of braids is more natural for constructing quantum invariants of knots and links and of c.c.o. 3-manifolds, since for 3-manifolds one needs to discard orientation. Then, the construction of a quantum invariant on a knot diagram is done by ‘reading off’ the value on the Morse diagram: one assigns linear maps, based on a solution to the Yang-Baxter equation, on the horizontal levels of the diagram (\cite{turaev1988yang}, \cite{reshetikhin1991invariants}, see also \cite{Garnerone2007QuantumAB} for a plat example). 

A lot of work has been done using the plat representation. For instance, in \cite{bigelow2002homological} Bigelow studied the possibility of computing the Jones polynomial of a link in terms of the action of a braid, representing the link as plat closure, over a homological pairing defined on a covering of the configuration space of \(n\) points into the \(2n\)-punctured disc. This result was analyzed in recent years also from a computational point of view \cite{Garnerone2007QuantumAB, jordan2008estimating} and in order to compute other link invariants, such as the Conway potential function \cite{yun2011braid} or other polynomial invariants \cite{garnerone2006quantum} and equivalence classes of links \cite{cavicchioli2021algorithmic}. 

From the above we derive that it would be very meaningful and useful to have in hands an algorithm for passing from one type of closure to the other and vice versa. Indeed, the first goal of this article is to give an algorithmic way to pass from the plat closure representative of a link in \(\mathbb{R}^3\) to its representation as the standard closure of a braid and vice-versa. See Section~\ref{section:main_result}. The passage from standard closure to plat closure is straightforward, following the isotopy depicted in the rightmost part of Fig.~\ref{fig:both_closure}. 
 The  most difficult is the passage from plat to standard closure. For this we introduce the \textit{\(S\)-move} for braiding upward oriented arcs in crossings, view Fig.~\ref{fig:s_move_intro}. The \textit{\(S\)-move}  is a combination of Birman's stabilization move to the case of oriented plats and Lambropoulou's braiding move (\cite{LaPhD}, see also \cite{lambropoulou1997markov}) for an alternative proof of the Alexander theorem and for extending the result to other 3-manifolds. 
\begin{figure}[h!]
    \centering
    \includegraphics[width = .3\textwidth]{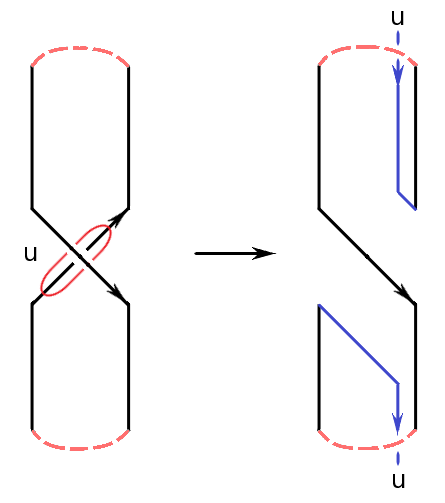}
    \caption{An abstraction of the \(S\)-move. }
    \label{fig:s_move_intro}
\end{figure}
\noindent For both directions we give closed algebraic representations, which are used in the implementation of the algorithm solving the problem.

Furthermore, the  study of links in handlebodies, thickened surfaces and c.c.o. 3-manifolds is  a fertile soil of research, with many researchers generalizing the results obtained in \(\mathbb{R}^3\) in the more general setting of 3-manifolds. 

We recall that a c.c.o. 3-manifold can be constructed via surgery along a framed link in $S^3$ and also via `gluing’ of two handlebodies (Heegaard decomposition). The surgery description of a c.c.o. 3-manifold is naturally associated to the standard closure of braids, while the handlebody decomposition is naturally associated to the plat closure (because representation of links via plat closures of braids in c.c.o. 3-manifolds uses braids in thickened surfaces as the starting point \cite{bellingeri2012hilden, cattabriga2018markov}, and a thickened surface can be viewed as the interface of the two handlebodies). Further, using Heegaard decompositions, in \cite{doll1993generalization} Doll introduced the notion of $(g,b)$-decomposition or generalized bridge decomposition for  links in a c.c.o. \(3\)-manifold, opening the way to the study of links in 3-manifolds via surface braid groups and their plat closures \cite{bellingeri2012hilden, cattabriga2018markov}. 

Because of the above, in the second part of the paper we extend our algorithms of transitioning from the one closure of braids to the other closure  to the cases of knots and links in handlebodies and thickened surfaces. See Sections~\ref{Handlebody} and~\ref{thickened}. We are especially interested in this last case, as another goal would be to extend our algorithms to the case of c.c.o. 3-manifolds and analyze also the action of this translation in terms of braid equivalences for isotopic knots and links. 

We also produce a C++ program that implements our algorithms; the program has linear computational complexity on the number of crossings and loop generators when passing from standard to plat closure, while it has quadratic computational complexity on the number of crossings when passing from plat to standard closure. 

The paper is structured as follows. In Section \ref{section:main_result} we present concisely our main results. 
In Section~\ref{Preliminaries} we  recall some notions and results about standard closure of braids and plat closure of braids. 
In Section \ref{from_standard} we describe the first part of the main result, passing from the representation with standard closure to plat closure. 
In Section \ref{construction} some useful constructions, needed for the proof of the main result, are described. 
Section \ref{main_result} contains the proof of the main result of the article, that is the algorithm to pass from the representation with plat closure to the one using standard closure in the setting of links in \(\mathbb{R}^3\). 
In Sections \ref{Handlebody} and \ref{thickened} the analogous results are constructed for links in handlebodies and thickened surfaces. 
Finally, in Section \ref{Algorithm} we present the main steps of the C++ program which reproduces the main results. 

Our results could apply to enhance the computation of link invariants, such as the Jones polynomial, for knots and links in 3-manifolds, and that, in general, they could have important implications in the fields of applied and computational topology.

\section{Main results}\label{section:main_result}

If we take an element of the braid group on \(2n\) strands, \(\mathcal{B}_{2n}\), and close it via the standard closure, we will  obtain a link, which in general will not be isotopic to the link that we would obtain using the plat closure. As stated in the Introduction, each representation  has its own advantages in theory and applications, so it would be useful to understand how to pass from one representation to the other. This is precisely the subject of this paper. Our result is, concisely, the following:

\begin{theorem}\label{teo_main_generale}
Given a link \(L\), in \(\mathbb{R}^3\) (or \(S^3\)), in a handlebody \(H_g\) or in a thickened c.c.o. surface of genus \(g\), \(N(\Sigma_g)\), represented by a braid \(B\) with plat closure, it is possible algorithmically to generate a braid \(\beta\) which represents an equivalent link but with standard closure and vice-versa. The algorithm from plat closure to standard closure has a computational complexity of \(O(N^2)\) where \(N\) is the number of crossings present in the braid \(B\), while it has a computational complexity of \(O(M)\), where \(M\) is the number of crossings present in the braid \(\beta\), when going from standard to plat closure. 
\end{theorem}

As we shall see, given a link represented by a braid in standard closure in any setting, it is straightforward to produce a braid in plat form representing the same isotopy class of the link. It is more elaborate to find the opposite one, passing from plat to standard closure of a braid. In what follows we will prove  Theorem~\ref{teo_main_generale} separately in every topological setting. We start from the easy way round, then, defining some useful constructions, we obtain the results for the three settings of \(\mathbb{R}^3\), handlebodies and thickened surfaces in sections \ref{main_result}, \ref{Handlebody} and \ref{thickened} respectively. 

We will usually employ capital letters for braids in plat form and small Greek letters for braids with standard closure.

\section{Preliminaries} \label{Preliminaries}

\subsection{Standard closure of braids}

Let \(\mathcal{B}_{n}\) be the Artin braid group \cite{artin1947theory}. It is defined as the fundamental group of the unordered configuration space of \(n\) points on the real plane, i.e. a group whose elements are equivalence classes of \(n\) paths \(\Psi = (\psi_1, \dots, \psi_n), \psi_i : [0,1] \rightarrow \mathbb{R}^3\), going from a set of \(n\) points of the plane to itself, and never intersecting (\(\psi_i(t) \neq \psi_j(t), \forall t \in [0,1], i,j\in {1, \dots, n}, \ i<j\)). The $\psi_i$'s are the {\it strands} of the braid. The middle illustration of Fig.~\ref{fig:both_closure}
 shows an example of a braid in \(\mathcal{B}_{4}\). 
 
    In the classical presentation, \(\mathcal{B}_{n}\) is generated by elementary braids \(\sigma_i, \ i\in {1, \dots, n-1}\), which exchange strands \(i\) and \(i+1\) as in Fig.~\ref{fig:braid_sigma_i}, with relations: 
\begin{align*}
   \sigma_i \sigma_j & = \sigma_j \sigma_i, \qquad \text{if } |i-j|\geq 2\\
   \sigma_i \sigma_{i+1} \sigma_i & = \sigma_{i+1} \sigma_i \sigma_{i+1}. \\
\end{align*}
For further information concerning the braid group, the reader could refer to \cite{birman2005braids, kassel2008braid}. 

\begin{figure}[h!]
    \centering
    \includegraphics[width = .55\textwidth]{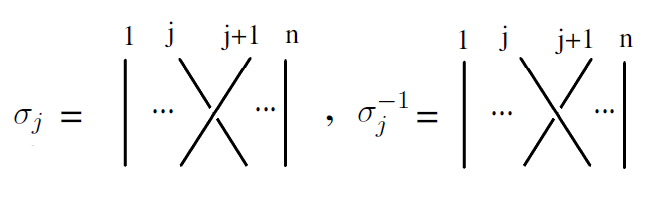}
    \caption{The \(\sigma_j\) and \(\sigma_j^{-1}\) generators.}
    \label{fig:braid_sigma_i}
\end{figure}



In order for a braid to represent a link we have to `close' it, for example identifying the starting set of points on the plane with itself. This operation can be viewed as connecting the top endpoints to their corresponding bottom endpoints with simple unlinked arcs, like in the right-hand illustration of Fig.~\ref{fig:both_closure}. This is called \emph{standard closure} of a braid. It is important to note that this closure naturally gives rise to an oriented link: indeed, if we consider the strands of the braid to be oriented downward, then all closing arcs will be oriented upward and the whole orientation is coherent. Conversely, by results of H.~Brunn \cite{brunn1892uber} and J.W.~Alexander \cite{alexander1923lemma}: 
\begin{theorem}[The Alexander theorem]
Every oriented link in \(\mathbb{R}^3\) can be represented as the standard closure of a braid in some \(\mathcal{B}_{n}\). 
\end{theorem}

Of course, different braids (also with a different number of strands) can be closed to give the same (oriented) link up to isotopy. A theorem by A.A.~Markov \cite{markov1935freie} states that  two braids have isotopic closures if and only if they can be connected by a finite sequence of conjugations and stabilization moves ($
\beta \in \mathcal{B}_{n} \sim \beta \sigma_{n}^{\pm{1}} \in \mathcal{B}_{n+1}$).  
In \cite{lambropoulou1997markov} another version of the Markov theorem is proved, with only one equivalence move: the \(L\)-move. 

The Markov theorem opened the possibility of studying the problem of knot and link classification via the Artin braid groups, and this possibility was exploited successfully by V.F.R.~Jones with the discovery of the homonymous Jones polynomial.




\subsection{Plat closure of braids}

One can also close a braid in a different way. Suppose that the braid has an even number of strands, i.e. it belongs to some \(\mathcal{B}_{2n}\). One can close the open ends by connecting adjacent top and bottom endpoints using simple, unlinked arcs (see left-hand illustration of Fig.~\ref{fig:both_closure}). This procedure is called \textit{plat closure} of a braid, see \cite{hilden1975generators}, \cite{birman1976stable} and references therein. 

It is important to note that the plat closure of a braid is not compatible with the natural orientation of the braid, as it was the case with the standard closure, since the short closing arcs would connect two downward strands. So, the result of  the plat closure of a braid is an unoriented knot or link, which will be referred to as \textit{plat}.

In \cite{birman1976stable} J.~Birman proves the analogue of the Alexander theorem for the plat closure, that is: 

\begin{theorem}[Birman] \label{platbraiding}
Every link can be represented as the plat closure of a braid in some \(\mathcal{B}_{2n}\). 
\end{theorem}

The idea of the proof is simple: in a projection plane equipped with top to bottom direction, consider a knot or link diagram with no horizontal arcs. Then pull by an isotopy every local maximum to the top and every local minimum to the bottom. The result will be a plat. See also  Theorem~\ref{alternativeplatbraiding} for an alternative proof.

We further point out that, one can extend the notion of plat closure also to braids with an {\it odd number of strands} as follows: 

\begin{definition}\label{odd_extension_R3} \rm 
Let \(\beta \in \mathcal{B}_{2n-1}\). We consider \({\beta}'\) to be the image of \(\beta\) under the natural embedding $\phi_{2n-1} : \mathcal{B}_{2n-1} \hookrightarrow \mathcal{B}_{2n}$, and we define the plat closure of \(\beta\) as the plat closure of \({\beta}'\) with the procedure described above. In practice, we add an extra strand in the far right of the braid \(\beta\). 
\end{definition} 
\noindent In the above definition, one could consider, analogously, the embedding $\phi_{i}$ of \(\mathcal{B}_{2n-1}\) into \(\mathcal{B}_{2n}\), whereby an extra free strand is added in the braid between the $i$th and the $(i+1)$st strands. This extra strand could run equally entirely over or entirely under the rest of the braid, as the two resulting plats would be isotopic, see Fig.~\ref{fig:pass_in_front}. We point out, though, that the plats resulting from different embedding maps will, in general, be non-isotopic.  

Since its importance in the equivalence theorem of links represented by plat closure of braids,  we  recall a subgroup of the braid group, the {\it Hilden braid subgroup}  \cite{hilden1975generators}, \(K_{2n}\), defined  as the subgroup of \(\mathcal{B}_{2n}\) generated by equivalence classes of homeomorphisms of \(\mathbb{R}^3_+\) leaving the set of closure arcs invariant on its boundary. 
In \cite{hilden1975generators} Hilden proves that \(K_{2n}\) is finitely generated, and in \cite{birman1976stable} Birman finds the presentation with the smallest number of generators and in terms of elements of $\mathcal{B}_{2n}$: 
 $$
 \{\sigma_1; \ \sigma_2 \sigma_1^2 \sigma_2; \ \sigma_{2i} \sigma_{2i-1} \sigma_{2i+1} \sigma_{2i}, \ 1\leq i \leq n-1\}.
 $$

\noindent Furthermore, Birman also proves in \cite{birman1976stable} an analogous result to the Markov theorem, concerning equivalence between plat closures of braids, namely two braids, with \(2n\) and \(2m\) strands, represent equivalent links via plat closure if there exist \(t \geq \max(n,m)\) such that the two elements: \[B_i' = B_i \sigma_{2n_i} \sigma_{2n_i+2} \dots \sigma_{2m-2} \in \mathcal{B}_{2m}, \qquad m > t, \> i = 1, 2, \> \] are in the same double coset of \(\mathcal{B}_{2m}\) modulo the Hilden subgroup \(K_{2m}\). 

\section{The \texorpdfstring{$\mathbb{R}^3$}{TEXT} case}

\subsection{From standard closure to plat closure of braids} \label{from_standard}

Let \(L\) be a link in \(\mathbb{R}^3\) and \(\beta \in \mathcal{B}_{m}\) be a braid which represents \(L\) via standard closure. We want to produce a new braid \(B \in \mathcal{B}_{2n}\) which represents the same link but via the plat closure. It is easy to produce \(B\) from \(\beta\): in fact, it can be done by pulling every \(i\)-th closing arc by isotopy  to the left and behind the braid up to the right of the \(i\)-th strand, becoming the \((2i-1)\)-th and \(2i\)-th strands respectively in the plat representation. Fig.~\ref{fig:L-C} is an abstraction of the transformation, while the last two illustrations of  Fig.~\ref{fig:both_closure} demonstrate an example of our trick. 

\begin{figure}[h!]
    \centering
    \includegraphics[width = .65\textwidth]{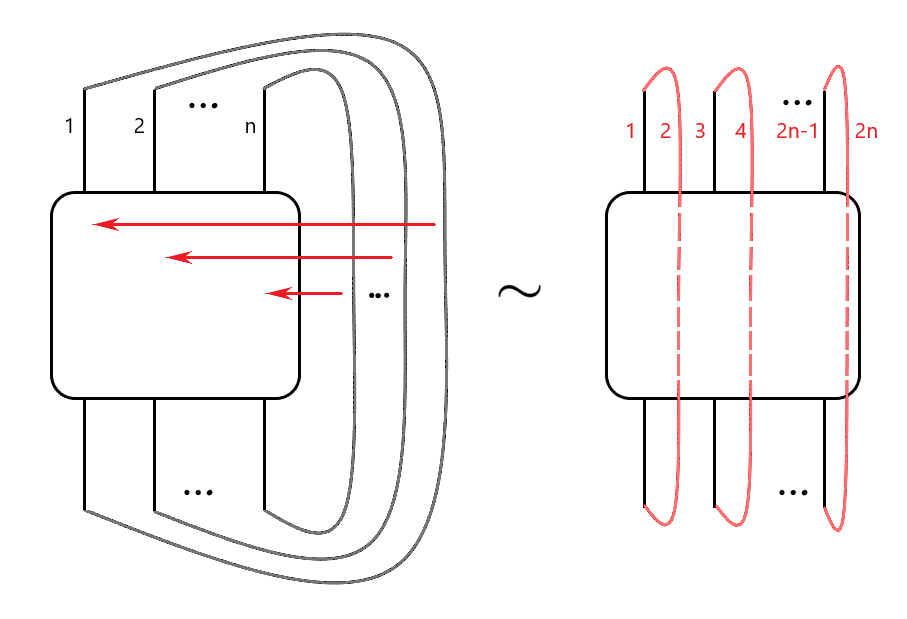}
    \caption{Passing from standard closure to the canonical plat form.}
    \label{fig:L-C}
\end{figure}


Note that, given the standard orientation of a classical braid, all the new strands from the closure are oriented upward (the red arcs in Fig.~\ref{fig:L-C}). 

In algebraic terms, it is possible to obtain the transformation replacing each \(\sigma_j^k\) by: 
\begin{equation}\label{eqn:transformation}
    \sigma_j^{k} \rightarrow \sigma_{2j-1}^{-1} \sigma_{2j}^k \sigma_{2j-1}, \qquad k = \pm 1.
\end{equation}
Note also that, in the process described above, a closing arc could be pulled equally well {\it over} the rest of the braid. 
We shall now define a useful notion: 

\begin{definition}\label{def:canonical_plat_form}
We will say that a braid \(B\) representing a link via plat closure is in \emph{canonical plat form} if every pair of even numbered endpoints are joined by a simple arc running entirely underneath the rest of the plat diagram, as abstracted in Fig.~\ref{fig:L-C}, and if \(B\) admits an orientation for which  these arcs are oriented upwards. Recall the example of the rightmost illustration in Fig.~\ref{fig:both_closure}.
\end{definition}
We can now state the following, which comprises an alternative proof of Theorem~\ref{platbraiding} and even sharpens it: 

\begin{theorem}\label{alternativeplatbraiding}
Every link \(L\) is isotopic to a closed braid in canonical plat form. Moreover, if \(\beta\) is a braid representing \(L\) via standard closure, the algorithm from standard closure to plat closure has a computational complexity of \(O(N)\), where \(N\) is the number of crossings in the braid \(\beta\).
\end{theorem}

\begin{proof} This is easily proven. We first assign an orientation to each link component of \(L\) and apply the Alexander theorem. Let \(\beta \in \mathcal{B}_n\) be a resulting braid with standard closure isotopic to the (oriented) link \(L\). Then we apply the isotopy illustrated in Fig.~\ref{fig:L-C}, which results in the  braid, say \(B \in \mathcal{B}_{2n}\), closed in canonical plat form.   Furthermore, since  \(\beta\)  has a natural orientation from top to bottom, all the red arcs in Fig.~\ref{fig:L-C} are oriented upwards and connect precisely the corresponding even numbered endpoints of \(B\). Hence, the orientation of the plat braid \(B\) is not the top to bottom orientation but the one inherited from the link. 
From the above, and dropping the orientation for \(L\), we can always turn a link into a plat in canonical form, according to Definition~\ref{def:canonical_plat_form}.

For this direction, from standard closure to plat closure,  equation (\ref{eqn:transformation}) gives the precise algebraic transformation for each braid generator. Note that, since the substitution is done once for each crossing, the algorithmic complexity is linear in the number of crossings.
\end{proof}


Furthermore we have: 
\begin{lemma}\label{prop:canon_plat}
A braid \(B \in \mathcal{B}_{2n}\) in canonical plat form is equivalent to a braid \(B' \in \mathcal{B}_{2n}\) in plat form admitting an orientation for which every pair of endpoints \((2i-1, 2i), i\in {1, \dots, n}\) there exists at least a simple arc running entirely above or entirely below the rest of the braid, and oriented upward, which joins the corresponding pair of endpoints. 
\end{lemma}
Since \(B\) is in canonical plat form, following the procedure illustrated in Fig.~\ref{fig:pass_in_front} and Fig.~\ref{fig:pass_top_bottom} the proof of Lemma \ref{prop:canon_plat} is straightforward. Note that this is possible both if the even arc is entirely over (left) or entirely under (right) the rest of the braid.


\begin{figure}[h!]
    \centering
    \includegraphics[width = .9\textwidth]{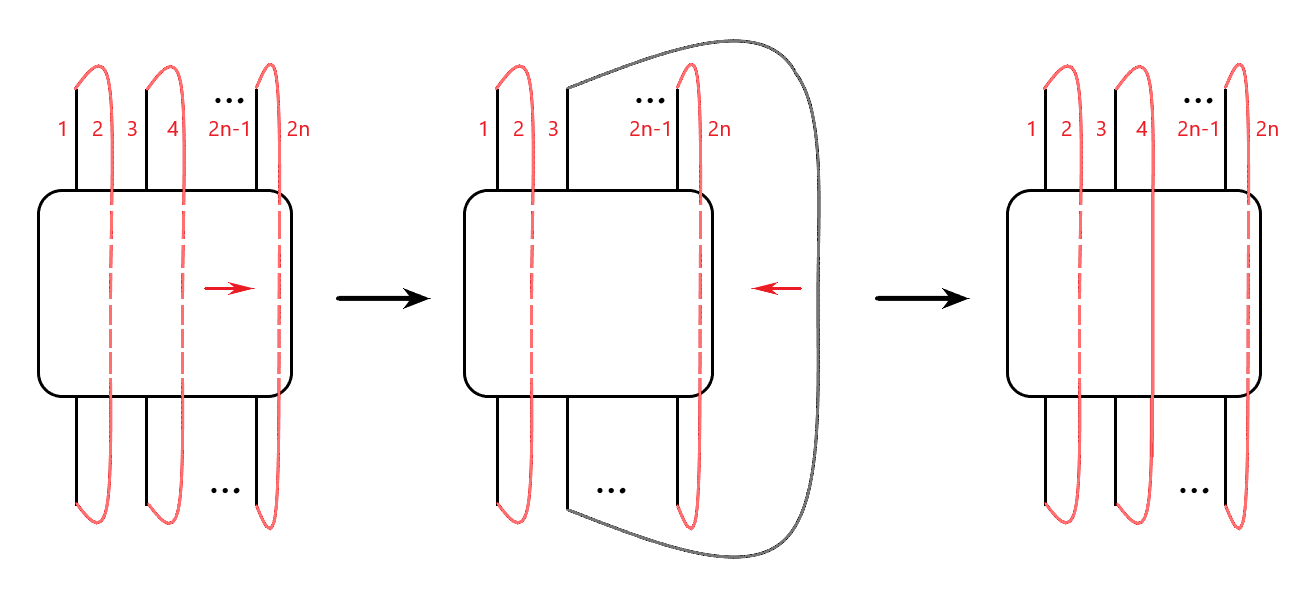}
    \caption{Passing one under arc of the canonical form into an over arc.}
    \label{fig:pass_in_front}
\end{figure}

\begin{figure}[h!]
    \centering
    \includegraphics[width = .9\textwidth]{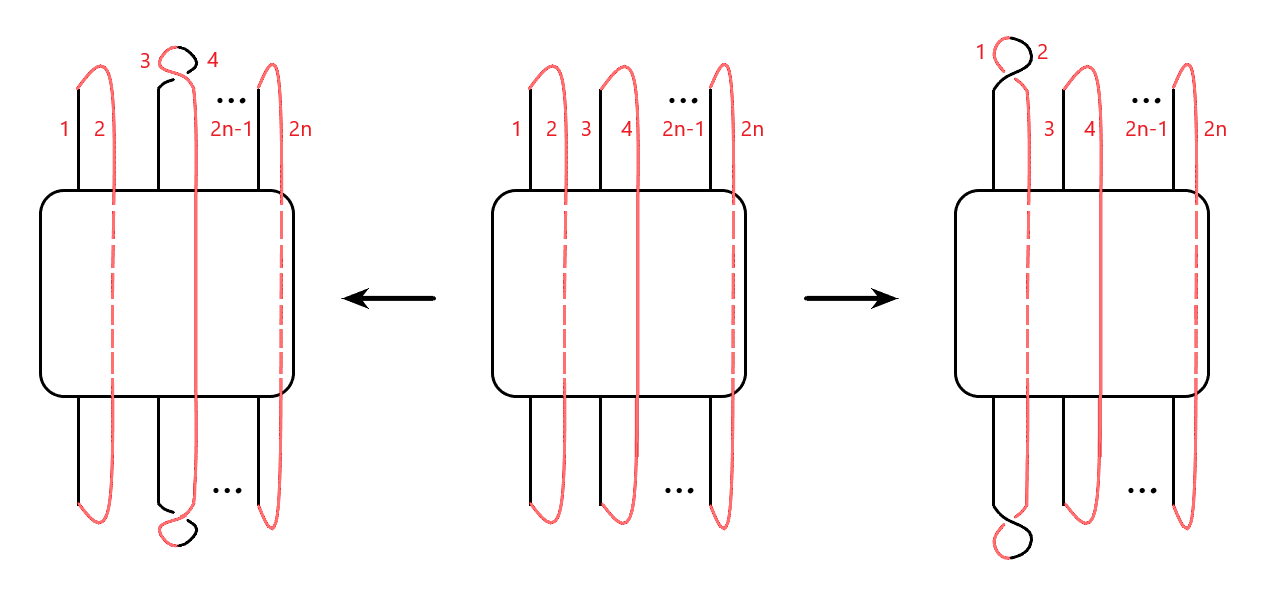}
    \caption{Passing one odd arc into an even arc, either entirely over (left) or entirely under (right) the rest of the braid.}
    \label{fig:pass_top_bottom}
\end{figure}

\subsection{From plat closure to standard closure of braids} \label{construction}

Let \(L\) be a link and \(B\in \mathcal{B}_{2n}\) a braid representing \(L\) via plat closure. Our aim is to prove the following. 

\begin{theorem}\label{th:main}
Given a link \(L\) represented by a braid \(B\in \mathcal{B}_{2n}\) with plat closure, it is possible  to generate algorithmically a braid \(\beta\) which represents an isotopic link but with standard closure. The algorithm from plat closure to standard closure has a computational complexity of \(O(N^2)\) where \(N\) is the number of crossings present in the braid \(B\). 
\end{theorem}

Our strategy for proving the direction from plat closure to standard closure is to generate from \(B\) a plat \(B'\) in canonical plat form with isotopic plat closure, and this will occupy the rest of this section. We shall first give some useful constructions. 

First, we give an orientation to the plat closure of \(B\in \mathcal{B}_{2n}\). 
  As a standard procedure, we give downward orientation to the leftmost strand of \(B\). Then we orient subsequent strands by following this orientation also in the plat closure arcs. If any strand of \(B\) lacks orientation (in the case of \(B\) representing a link with at least 2 components) then we take the first strand from the left with no orientation (it will be an odd strand) and give it a downward orientation. Of course, different plat orientations could result in different oriented link classes. Here we are interested only in the unoriented result of the algorithm, which will be always isotopy equivalent to the unoriented link class of the plat closure of \(B\).

Then we focus on the crossings. We divide them in four categories as illustrated in Fig.~\ref{fig:crossings}, where in the region of the red dot there is an over or an under crossing. We want to create a braid with all crossing of type \(a\). For this reason we shall apply \(S\)-moves to every crossing of type \(b, c, d\). 

\begin{figure}[h!]
    \centering
    \includegraphics[width = .7\textwidth]{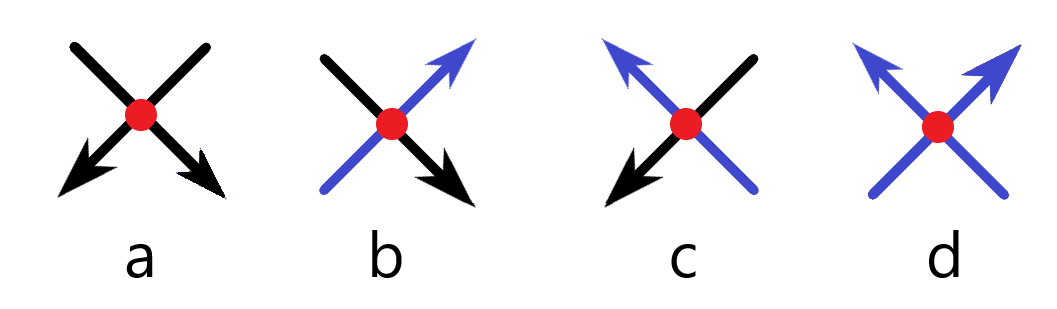}
    \caption{The four possible crossings. }
    \label{fig:crossings}
\end{figure}

\subsubsection{The \texorpdfstring{$S$}{TEXT}-move}

In \cite{birman1976stable} Birman introduced the \textit{stabilization move} as a move between braids in plat form with isotopic plat closure, see Fig.~\ref{fig:stabiliz_birman}. In the figure we see a small arc of the braid included in the disk, \([b_j', b_j, b_j'']\), to be isotopically replaced by three consecutive arcs, \([b_j', b_j^*, b_j^\#, b_j'']\), forming two spikes which reach the top and bottom part of the braid respectively. Here we use an open-ended version of a special case of the stabilization move; this shall be called ``\(S\)-move'', more precisely: 

\begin{figure}[h!]
    \centering
    \includegraphics[width = .55\textwidth]{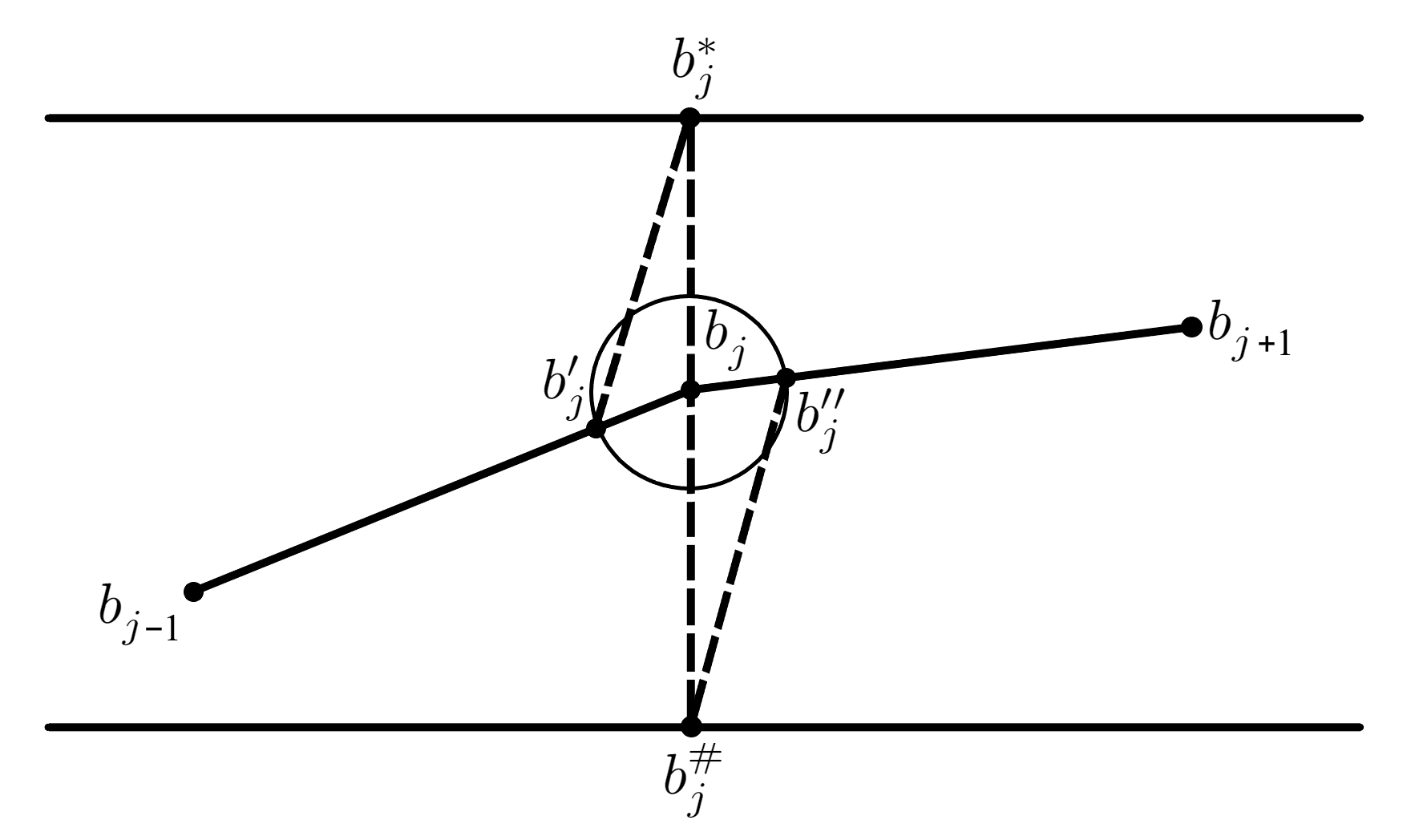}
    \caption{An abstraction of Birman's stabilization move.}
    \label{fig:stabiliz_birman}
\end{figure}

\begin{definition}[\(S\)-move]\label{def:s_move}
The \textit{\(S\)-move} is an adaptation of the stabilization move in the case of oriented plats, with the following features (view Fig.~\ref{fig:s_move}): 
\begin{itemize}
    \item it is applied on up-arcs; 
    \item the two isotopy spikes take place both entirely underneath or entirely above the rest of the diagram according to whether it is an under or over up-arc; 
    \item it eliminates the up-arc and produces two open-ended down-arcs which make a pair of corresponding braid strands. 
\end{itemize}
\end{definition} 
\begin{figure}[h!]
    \centering
    \includegraphics[width = .75\textwidth]{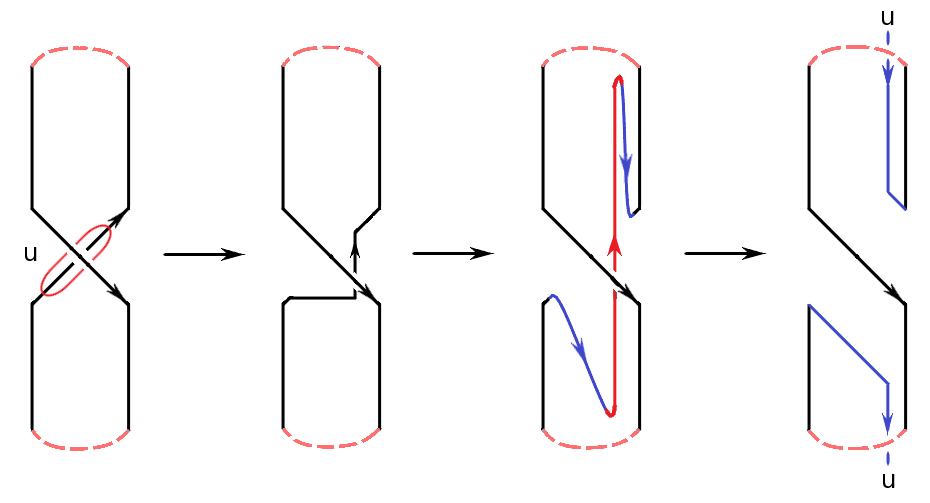}
    \caption{The \(S\)-move  with its intermediate steps. }
    \label{fig:s_move}
\end{figure}
\begin{remark} \rm 
It is worth pointing out that  the \(S\)-move is precisely the braiding move defined in  \cite{LaPhD} (see also \cite{lambropoulou1997markov}) for an alternative proof of the Alexander theorem and for extending the result to other 3-manifolds. 
\end{remark}
Take, for example, a \(b,c\) or \(d\)-type of crossing and consider a small piece of the up going strand (in case \(d\) we do the operation for both strands one after the other). Fig.~\ref{fig:s_move} illustrates the case of a type \(b\) crossing where the up-arc is an `under' arc. The \(S\)-move starts with isotoping this arc across a small sliding triangle so that it is replaced by the horizontal and vertical sides of the triangle. Then a stabilization move is applied to the vertical side of the triangle, so the two isotopy spikes take place both entirely underneath or entirely above the rest of the diagram. In Fig.~\ref{fig:s_move} this is depicted by the addition of two blue and one red arcs. The \(S\)-move is completed by omitting the red arc from the representation. 
In Fig.~\ref{fig:s_move_cases} are represented the three possible cases of crossings. In both figures, as well in subsequent figures, we have indicated with dust-grey lines the arcs that may appear in the plat closure. Notice that there could be instances where there are no closure arcs (when the left hand side strand is an even numbered one). 

\begin{figure}[h!]
    \centering
    \includegraphics[width = \textwidth]{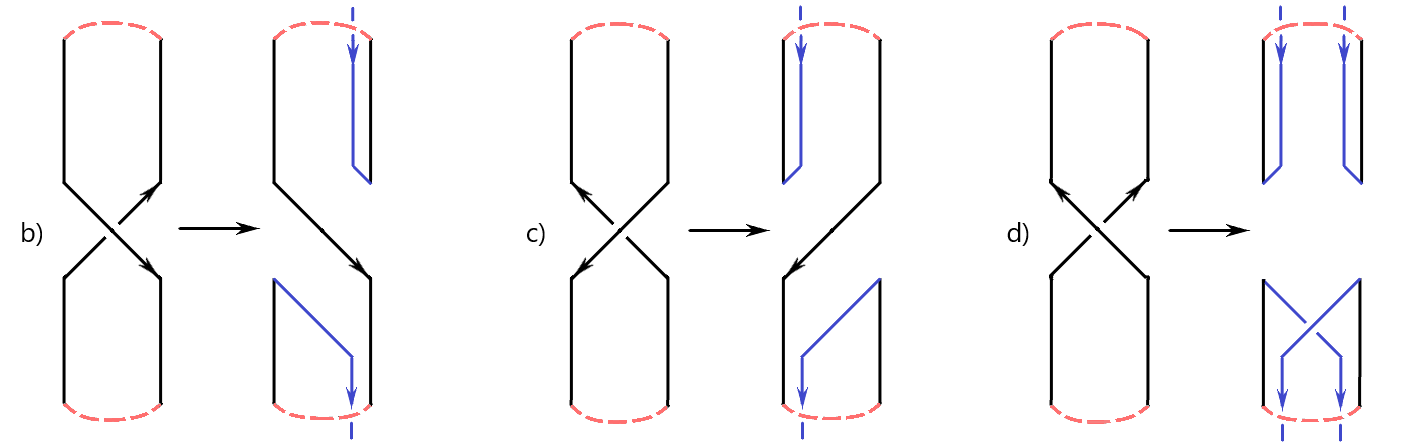}
    \caption{The three possible outcomes of a \(S\)-move. }
    \label{fig:s_move_cases}
\end{figure}

\begin{remark} \rm 
The elimination of the up-arc in Definition~\ref{def:s_move} is  crucial so that the application of an \(S\)-move will not create other \(b, c\) or \(d\)-type crossings. 
\end{remark}

\subsubsection{Managing crossings on the same column} 

If we have different \(S\)-moves applied on the same column, we do them in such a way that the new strands coming from the same type of crossings never intersect each other. An example can be seen in Fig.~\ref{fig:s_move_column}: the red strands coming from the \(c\)-type crossings are all on the left, one after the other, while the blue strands coming from the \(b\)-type crossings are all on the right in a similar manner. \(d\)-type crossings generate one left and one right strand. It is important to note that, once we have the sequence of all the crossings (with labelled types), we can give the univocal sequence of the new crossings generated by the algorithm. 

More precisely, suppose that we have, in the column \(i\), a total number of \(H\) type \(c\) and \(d\) crossings and a total number of  \(K\) type \(b\) and \(d\) crossings (the type \(d\) crossings are counted doubly because each one gives rise to a blue and a red strand). After an \(a, b, c\) or \(d\)-type crossing, say \(\bar{\sigma_i}\), is dealt with, we will have a situation like in Fig.~\ref{fig:type}. Here, we can see how all different cases present different outcomes for the starting and ending original strands defining the column~\(i\). The sequences of the final crossings can be summarized in these explicit formulae: 

\begin{flalign}
C_a(\bar{\sigma_i}) = & (\sigma_{i}^{\delta_i} \sigma_{i+1}^{\delta_{i+1}} \dots \sigma_{i+H-1}^{\delta_{i+H-1}})   (\sigma_{i+H+K}^{\delta_{i+H+K}} \sigma_{i+H+K-1}^{\delta_{i+H+K-1}} \dots \sigma_{i+H+1}^{\delta_{i+H+1}})   \bar{\sigma}_{i+H} \label{eqn_C_a} \\[1ex] 
\nonumber & \quad (\sigma_{i+H-1}^{-\delta_{i+H-1}} \sigma_{i+H-2}^{-\delta_{i+H-2}} \dots \sigma_{i}^{-\delta_{i}}) (\sigma_{i+H+1}^{-\delta_{i+H+1}} \sigma_{i+H+2}^{-\delta_{i+H+2}} \dots \sigma_{i+H+K}^{-\delta_{i+H+K}})\\[2ex]  
C_b(\bar{\sigma_i}) = & (\sigma_{i}^{\delta_{i}} \sigma_{i+1}^{\delta_{i+1}} \dots \sigma_{i+H-1}^{\delta_{i+H-1}})   (\sigma_{i+H}^{\delta_{i+H}} \sigma_{i+H+1}^{\delta_{i+H+1}} \dots \sigma_{i+H+K-1}^{\delta_{i+H+K-1}})   (\sigma_{i-1}^{\delta_{i-1}} \sigma_{i}^{\delta_{i}} \dots \sigma_{i+H-2}^{\delta_{i+H-2}}) \label{eqn_C_b} \\[2ex] 
C_c(\bar{\sigma_i}) = & (\sigma_{i+H+K}^{\delta_{i+H+K}} \sigma_{i+H+K-1}^{\delta_{i+H+K-1}} \dots \sigma_{i+H+1}^{\delta_{i+H+1}})   (\sigma_{i+H}^{\delta_{i+H}} \sigma_{i+H-1}^{\delta_{i+H-1}} \dots \sigma_{i+1}^{\delta_{i+1}}) \label{eqn_C_c} \\[1ex] 
\nonumber & \quad (\sigma_{i+H+K+1}^{\delta_{i+H+K+1}} \sigma_{i+H+K}^{\delta_{i+H+K}} \dots \sigma_{i+H+2}^{\delta_{i+H+2}}) \\[2ex] 
C_d(\bar{\sigma_i}) = & (\sigma_{i}^{\delta_{i}} \sigma_{i+1}^{\delta_{i+1}} \dots \sigma_{i+H-1}^{\delta_{i+H-1}})(\sigma_{i+H+K}^{\delta_{i+H+K}} \sigma_{i+H+K-1}^{\delta_{i+H+K-1}} \dots \sigma_{i+H+1}^{\delta_{i+H+1}})   \bar{\sigma}_{i+H}^{-1} \label{eqn_C_d}
\end{flalign}
with \(\delta_{i-1}, \delta_i, \delta_{i+1}, \dots, \delta_{i+H+K} \in \{+1, -1\}\) depending on the \(S\)-moves. \\

When dealing with this transformation is easy to check that we actually create some local minima or maxima along our braid. Checking all possibilities of subsequent \(a,b,c\) and \(d\)-type crossings, according to the orientation of the strands, the situation can be brought to a standard braid with a local isotopy move or a Reidemeister I move. In Fig.~\ref{fig:all_interesting_cases} there are some interesting cases. Note that, in fact, these are the only cases where Reidemeister I will appear, indeed when we consider braids with multiple columns, the process will eliminate all up arcs in the boundary of the columns with the process we describe. If we have a Reidemeister I move it is important to note that we have to eliminate the related crossing, following this scheme: 
\begin{itemize}
    \item two \(d\)-type crossings one after the other on the same column: remove \(\sigma_i^{\delta_i}\) and \(\sigma_{i+H+K}^{\delta_{i+H+K}}\) from the first \(C_d\); 
    \item a \(b\)-type crossing followed by a \(c\)-type crossing on the same column: remove \(\sigma_{i-1}^{\delta_{i-1}}\) from \(C_b\); 
    \item a \(c\)-type crossing followed by a \(b\)-type crossing on the same column: remove \(\sigma_{i+H+2}^{\delta_{i+H+2}}\) from \(C_c\); 
\end{itemize}

\begin{figure}[h!]
\centering
\begin{subfigure}{.45\textwidth}
  \centering
\includegraphics[width=\textwidth]{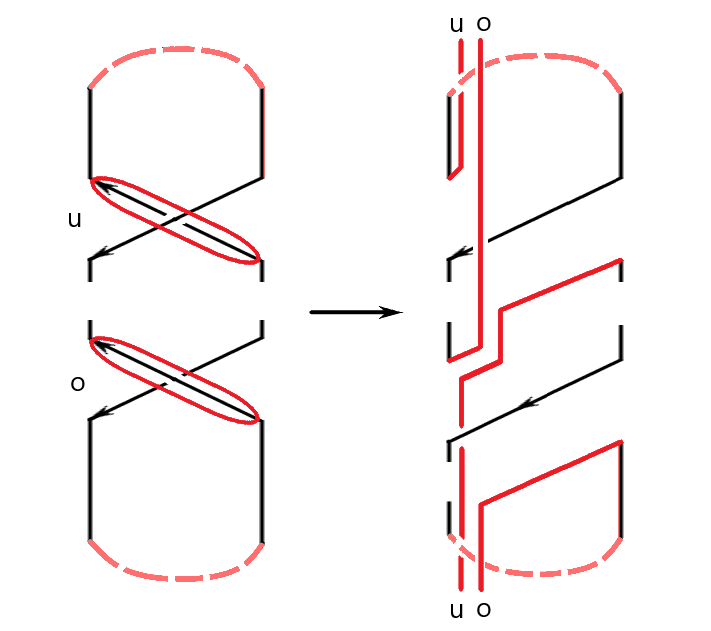}
\end{subfigure}
\begin{subfigure}{.45\textwidth}
  \centering
\includegraphics[width=\textwidth]{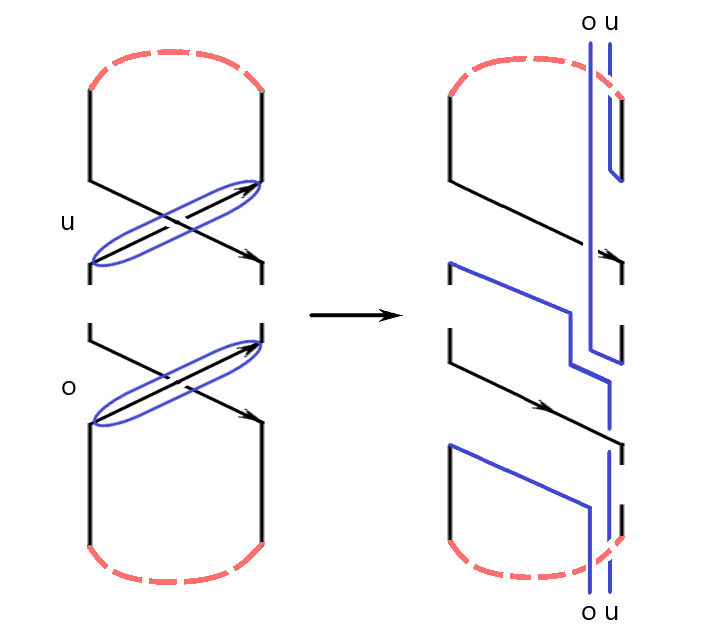}
\end{subfigure}
    \caption{The \(S\)-move applied on different crossings of same types on the same column (\(c\)-type and \(b\)-type respectively), which create only left or only right new strands. }
\end{figure}

\begin{figure}[h!]
    \centering
    \includegraphics[width = .5\textwidth]{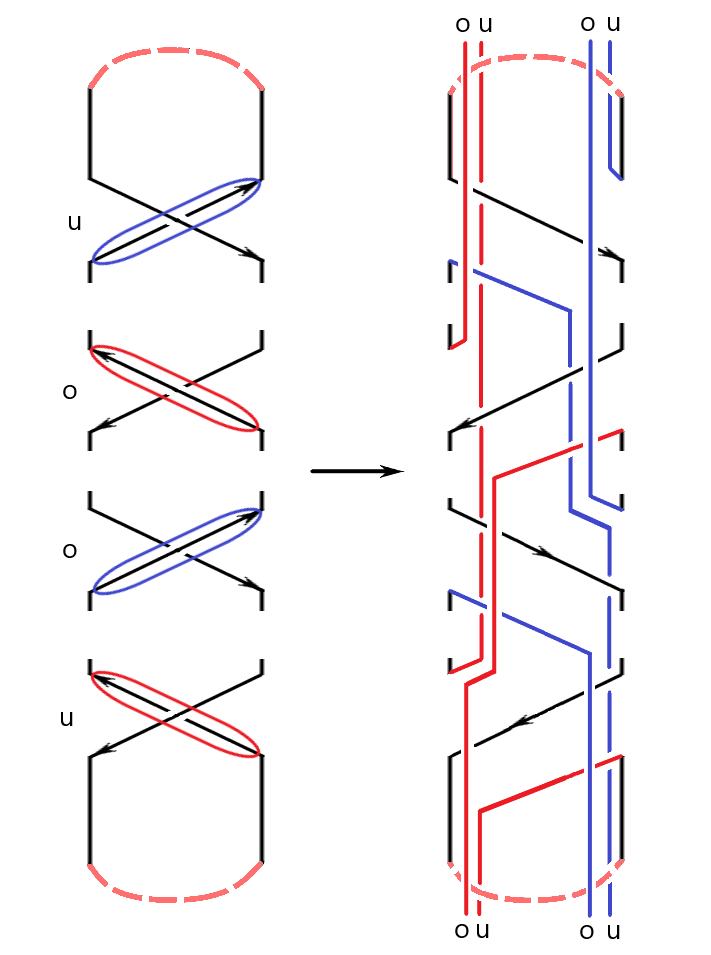}
    \caption{The \(S\)-move applied on different crossings of types \(c\) and \(b\) on the same column. }
    \label{fig:s_move_column}
\end{figure}

\begin{figure}[ht!]
\centering
\begin{subfigure}{.3\textwidth}
  \centering
\includegraphics[width=\textwidth]{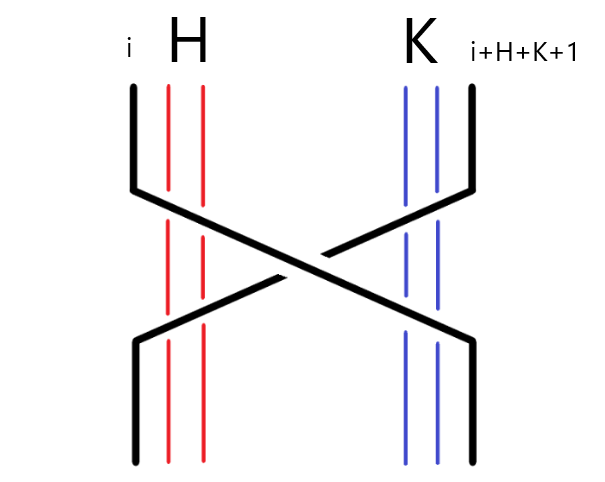}
  \caption{\(a\)-type}
\end{subfigure}
\begin{subfigure}{.3\textwidth}
  \centering
\includegraphics[width=\textwidth]{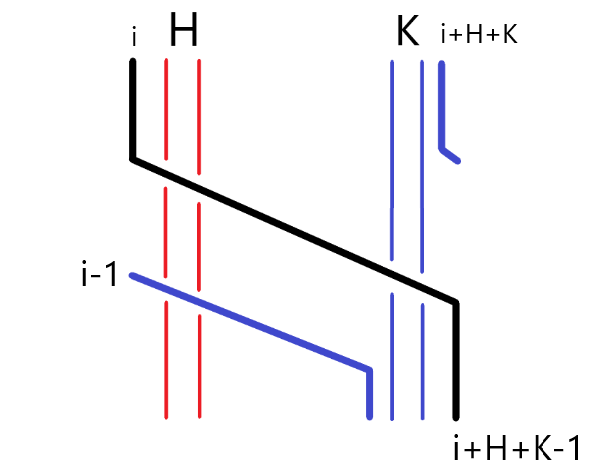}
  \caption{\(b\)-type}
\end{subfigure}

\begin{subfigure}{.3\textwidth}
  \centering
\includegraphics[width=\textwidth]{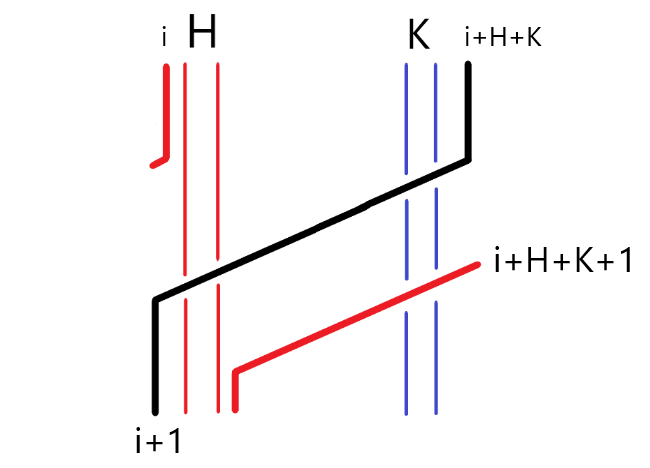}
  \caption{\(c\)-type}
\end{subfigure}
\begin{subfigure}{.3\textwidth}
  \centering
\includegraphics[width=\textwidth]{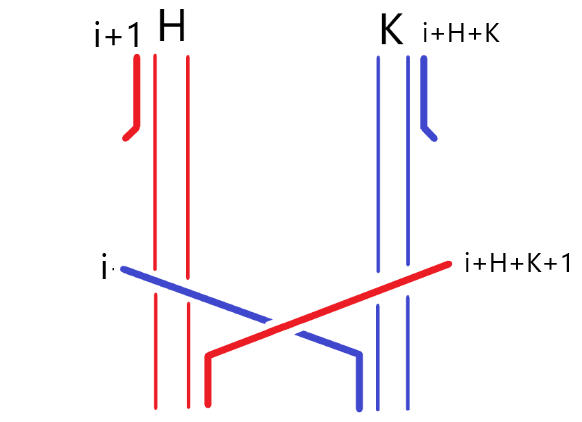}
  \caption{\(d\)-type}
\end{subfigure}
    \caption{The four possible outcomes after all the \(S\)-moves in a column. }
    \label{fig:type}
\end{figure}

\begin{figure}[ht!]
    \centering
    \includegraphics[width = \textwidth]{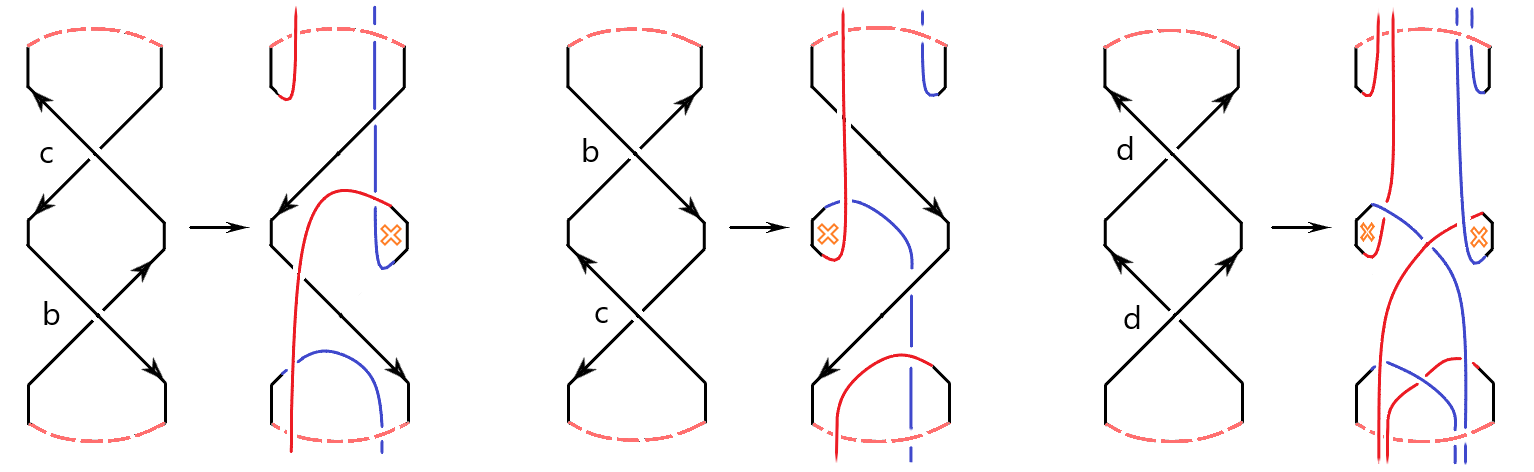}
    \caption{Some cases where the Reidemeister I move is needed to bring the braid into standard position. }
    \label{fig:all_interesting_cases}
\end{figure}

\subsubsection{Top part and bottom part management}\label{top_bottom_management}
Now that we have dealt with all the possibilities in the braid, we have to address the problem of the intersections of the new strands with the plat closure arcs of the braid. In fact the new strands and these intersections are formed with applications of \(S\)-moves, since the new strands extend higher than the top plat closure and lower than the bottom plat closure. 

For this reason, we have some possibilities occurring, depending on the orientation of the top (resp. bottom) plat closure arc and the possible crossings present right after (resp. above) it. 

Let us consider the top part first. Observing that only the odd labelled columns contain the plat closure arcs, we will only focus on these, because  we will not add crossings in the top part while performing the \(S\)-move along the even columns. We will address the odd column that we are working on with label \(i\). 
We have two possible orientations for the plat closure arc, from left to right and from right to left, and we will analyze them separately. 

With orientation from right to left, it is easy to see that, if there are crossings in columns \(i\) or \(i+1\), the first that joins the up arc of the closure arc one needs to agree to the orientation: type \(b\) or \(d\) in column \(i\), type \(c\) or \(d\) in column \(i+1\). 

Of course, in the first case, we will need to perform a \(S\)-move on the crossing, obtaining the situation depicted in Fig.~\ref{fig:top_part_1}. Considering the number of type \(c\) and \(d\) crossings as before, it is straightforward to check that in this case the new crossings form the word: 
\begin{equation}\label{eqn_t1}
T_{L, i} = \sigma_{i+H+K-2} \sigma_{i+H+K-3} \dots \sigma_{i+1} \sigma_i    
\end{equation}
with signs depending on the \(S\)-moves. It is important to note that, as we saw in the previous paragraph, we can apply a Reidemeister I move in the top right part to get the braid into a standard position: we would have also a \(\sigma_{i+H+K-1}\) but we have removed it. 

In the same way, for the \(i+1\) case we get the construction in Fig.~\ref{fig:top_part_2} and obtain the expression: 
\begin{equation}\label{eqn_t2}
T_{L, i+1} = \sigma_{i+H+K-1} \sigma_{i+H+K-2} \dots \sigma_{i+1} \sigma_i
\end{equation}
with signs depending on the \(S\)-moves, and no Reidemeister moves. 

With orientation from left to right, in the same way as before, we can see that if there are crossings in columns \(i\) or \(i-1\), the first one joining the up-arc of the closure arc needs to agree to the orientation: type \(b\) or \(d\) in column \(i-1\), type \(c\) or \(d\) in column \(i\). Analogously, following the construction in Fig.~\ref{fig:top_part_3}-\ref{fig:top_part_4}, we obtain the two expressions: 
\begin{equation}\label{eqn_t3}
T_{R, i-1} = \sigma_{i} \sigma_{i+1} \dots \sigma_{i+H+K-2} \sigma_{i+H+K-1}
\end{equation}
\begin{equation}\label{eqn_t4}
T_{R, i} = \sigma_{i} \sigma_{i+1} \dots \sigma_{i+H+K-1} \sigma_{i+H+K-2}
\end{equation}
with signs depending on the \(S\)-moves and a Reidemeister I move in the second case, which changes the number of the strings and leads to the final expression. 

\bigbreak

Now we consider the bottom part. Following the same process, we have the two possible orientations and four possible cases in total for the crossings: 
right to left, crossing of type \(c\) or \(d\) in column \(i-1\), Fig.~\ref{fig:bottom_part_1}; we can check how in column \(i-1\) nothing changes from the standard \(c\) or \(d\) case, but in column \(i\) we have 
\begin{equation}\label{eqn_b1}
B_{L, i-1} = \sigma_{i+H+K-1} \sigma_{i+H+K-2} \dots \sigma_{i}
\end{equation}

Right to left, crossing of type \(b\) or \(d\) in column \(i\), Fig.~\ref{fig:bottom_part_2}; following the picture, we can perform the Reidemeister 2 move on the left side (since all the crossings are coupled) eliminating the crossings on the left part and obtaining \begin{equation}\label{eqn_b2}
B_{L, i} = \sigma_{i+H} \sigma_{i+H+1} \dots \sigma_{i+H+K-2}
\end{equation}
with signs depending on the \(S\)-moves. 

Left to right, crossing of type \(c\) or \(d\) in column \(i\), Fig.~\ref{fig:bottom_part_3}; as before, we perform the Reidemeister 2 move but on the right side, obtaining 
\begin{equation}\label{eqn_b3}
B_{R, i} = \sigma_{i} \sigma_{i+1} \dots \sigma_{i+H-2}
\end{equation}
with signs depending on the \(S\)-moves. 

Left to right, crossing of type \(b\) or \(d\) in column \(i+1\), Fig.~\ref{fig:bottom_part_4}; as before, in column \(i+1\) nothing changes from the standard \(b\) or \(d\) case, adding in column \(i\) 
\begin{equation}\label{eqn_b4}
B_{R, i+1} = \sigma_{i} \sigma_{i+1} \dots \sigma_{i+H+K-1}
\end{equation}
with signs depending on the \(S\)-moves. 

\begin{figure}[h!]
\centering

\begin{subfigure}{.7\textwidth}
  \centering
  \includegraphics[width=\textwidth]{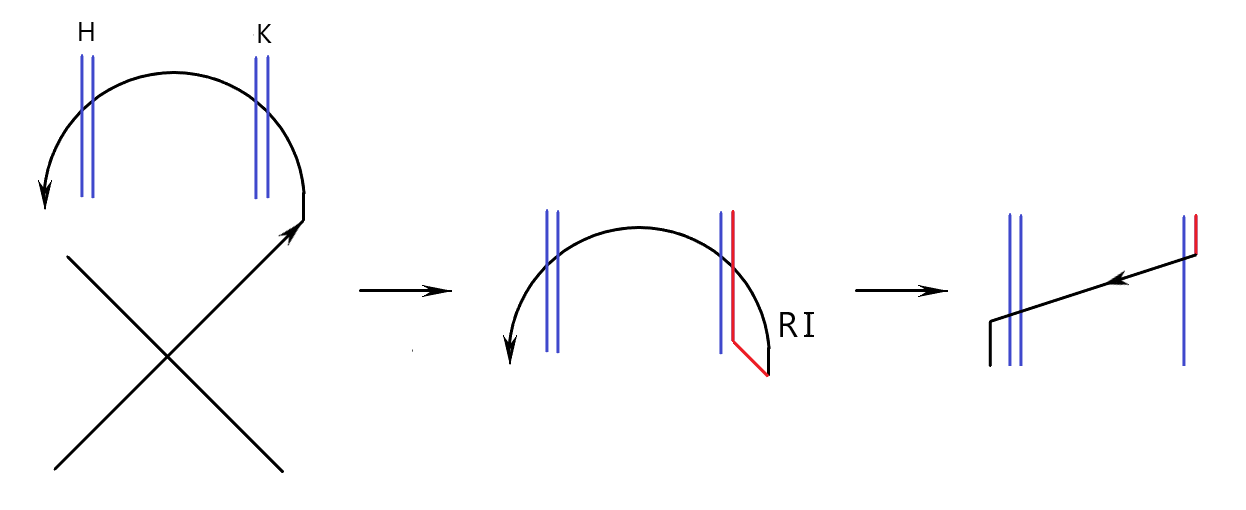}
  \caption{Top part with orientation right to left, first case;}
  \label{fig:top_part_1}
\end{subfigure}

\begin{subfigure}{.9\textwidth}
  \centering
  \includegraphics[width=\textwidth]{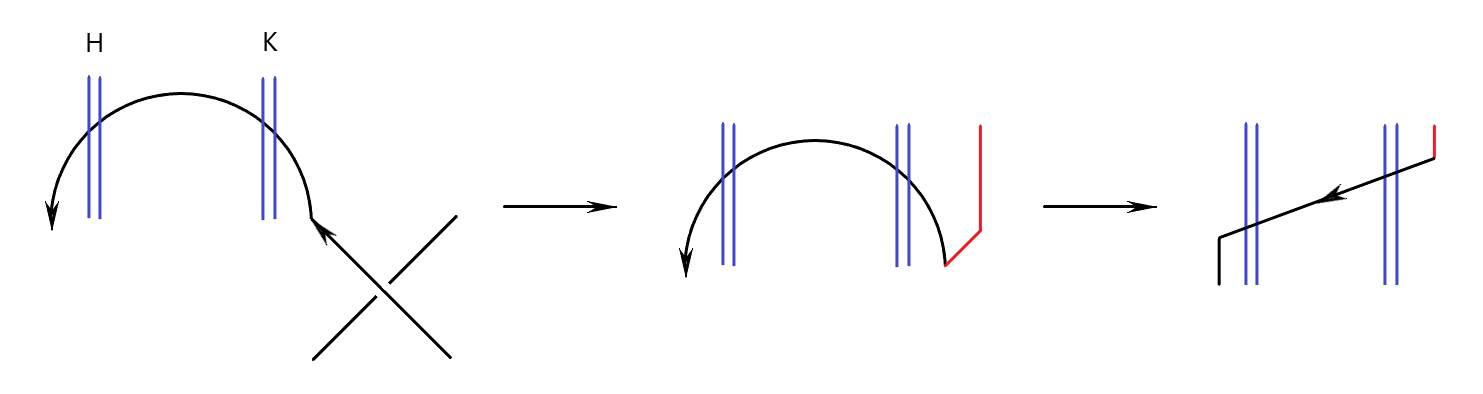}
  \caption{Top part with orientation right to left, second case;}
  \label{fig:top_part_2}
\end{subfigure}

\begin{subfigure}{.9\textwidth}
  \centering
  \includegraphics[width=\textwidth]{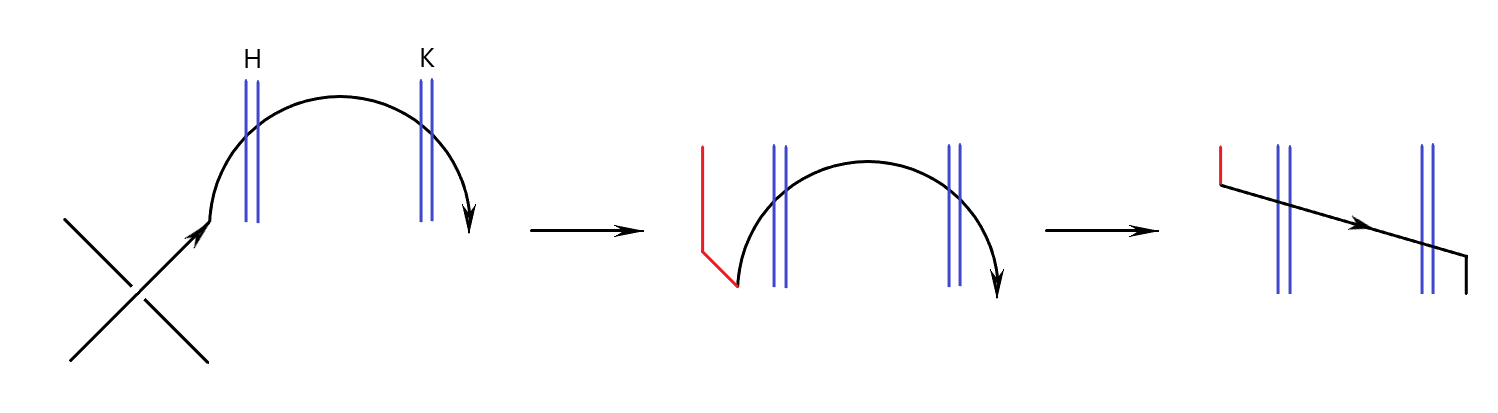}
  \caption{Top part with orientation left to right, first case;}
  \label{fig:top_part_3}
\end{subfigure}

\begin{subfigure}{.7\textwidth}
  \centering
  \includegraphics[width=\textwidth]{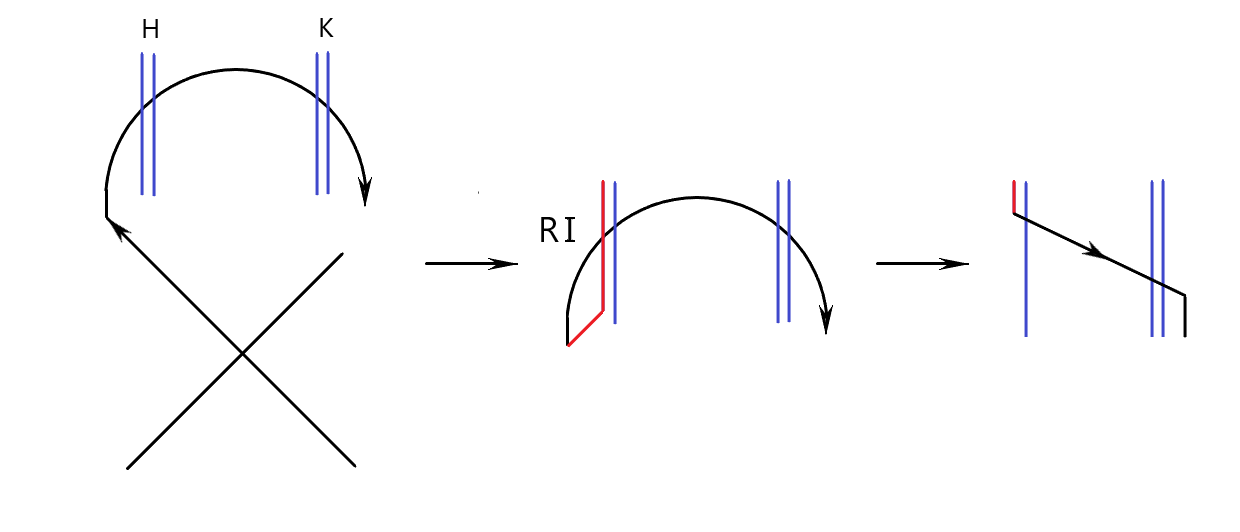}
  \caption{Top part with orientation left to right, second case;}
  \label{fig:top_part_4}
\end{subfigure}
\caption{Management of the top plat closure arcs. }
\end{figure}

\begin{figure}[h!]
\centering

\begin{subfigure}{.9\textwidth}
  \centering
  \includegraphics[width=\textwidth]{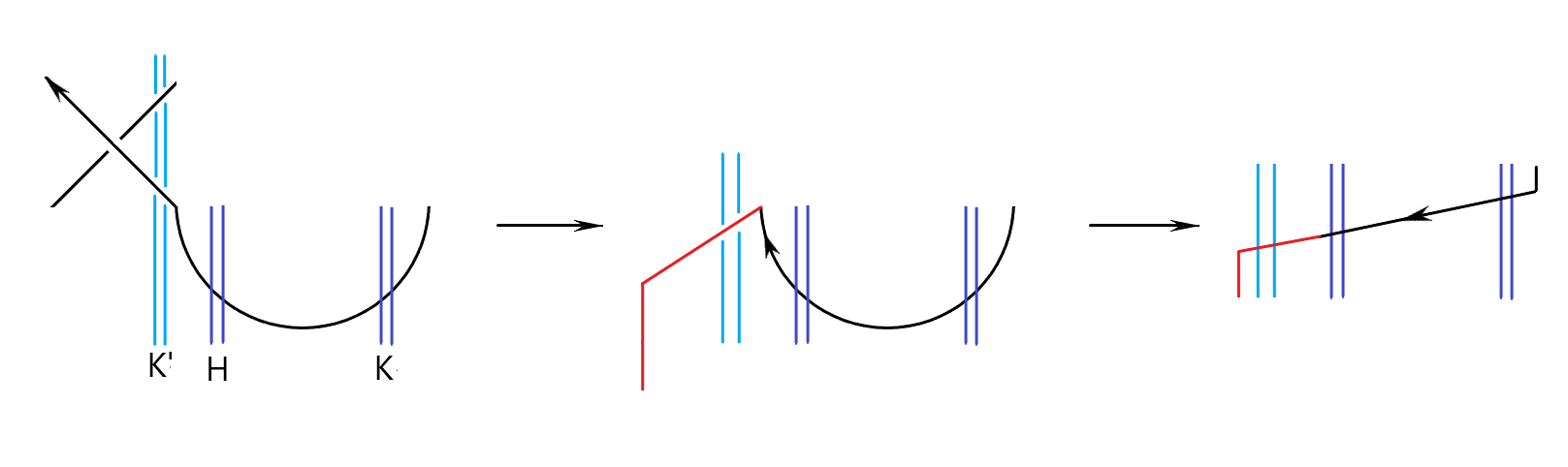}
  \caption{Bottom part with orientation right to left, first case;}
  \label{fig:bottom_part_1}
\end{subfigure}

\begin{subfigure}{.7\textwidth}
  \centering
  \includegraphics[width=\textwidth]{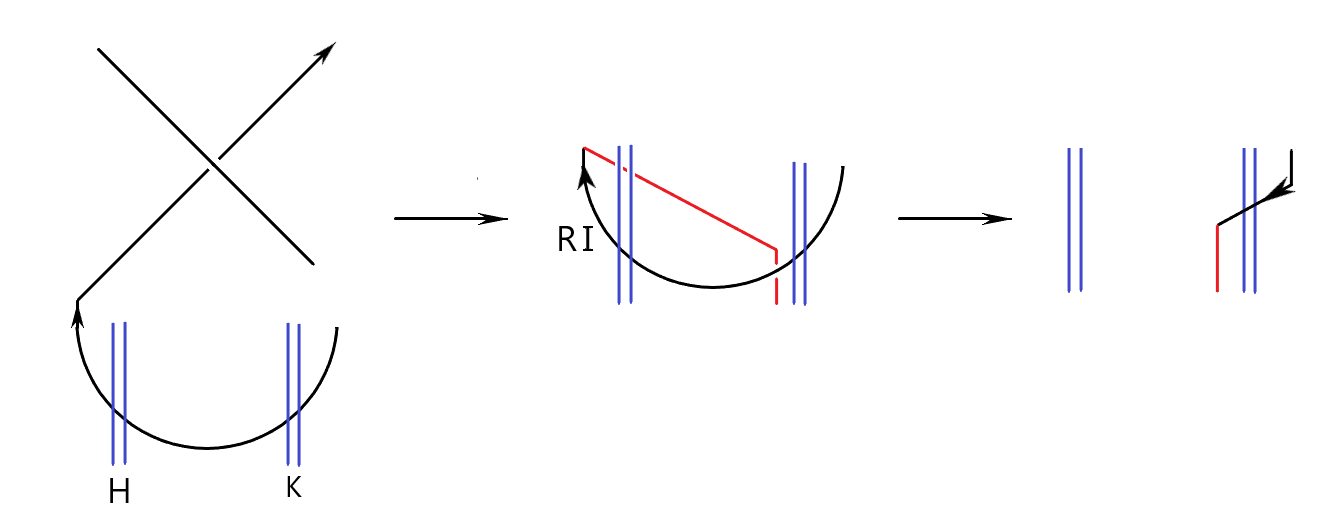}
  \caption{Bottom part with orientation right to left, second case;}
  \label{fig:bottom_part_2}
\end{subfigure}

\begin{subfigure}{.8\textwidth}
  \centering
  \includegraphics[width=\textwidth]{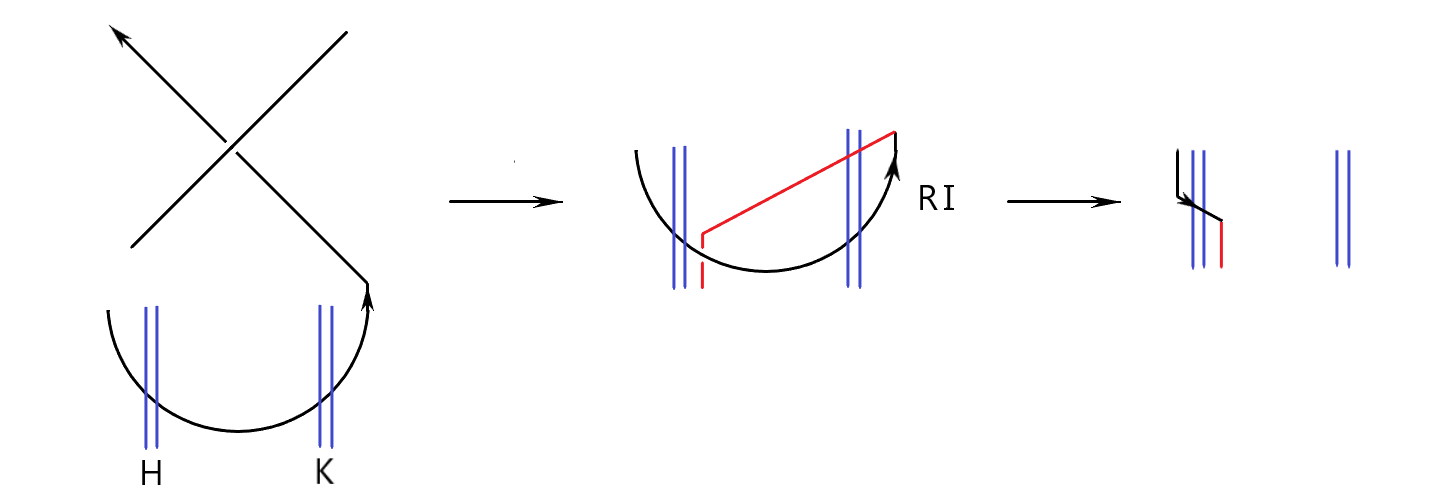}
  \caption{Bottom part with orientation left to right, first case;}
  \label{fig:bottom_part_3}
\end{subfigure}

\begin{subfigure}{.9\textwidth}
  \centering
  \includegraphics[width=\textwidth]{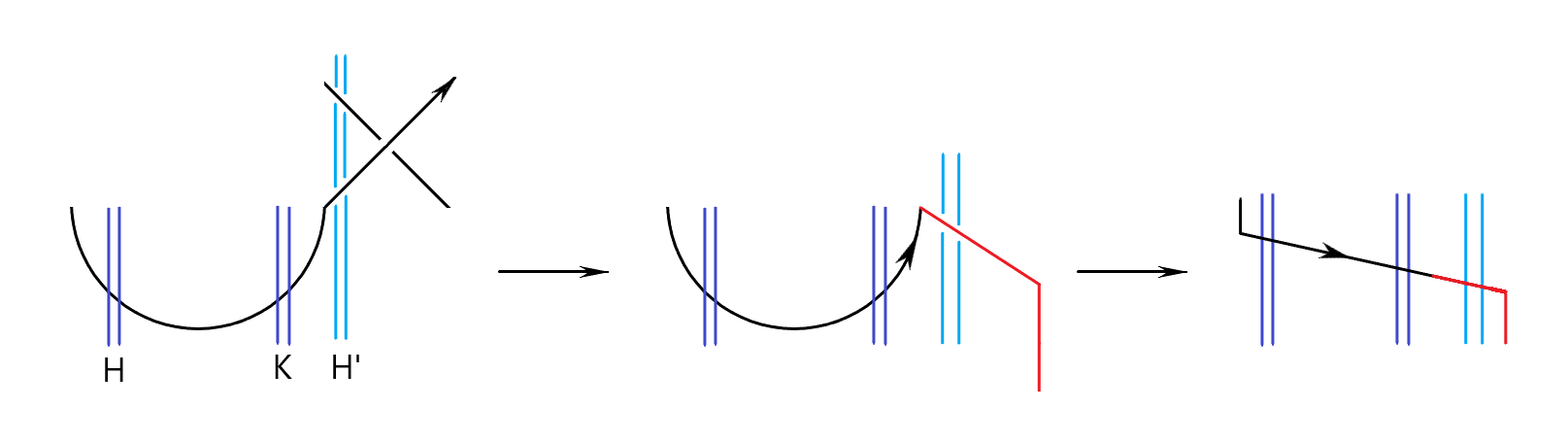}
  \caption{Bottom part with orientation left to right, second case;}
  \label{fig:bottom_part_4}
\end{subfigure}
\caption{Management of the bottom plat closure arcs.}
\end{figure}

\subsubsection{Trivial cases}

Until now we always considered the case of a column with at least one crossing. If we have no crossings in two subsequent columns then we will have between them a strand which goes straight from top to bottom. If it has a upward orientation then it can be considered as one of the closures in the standard closure, putting it aside. 

Moreover, if we have three consecutive empty columns, the first one of which is odd labeled, then we will have an extra trivial component to our link. 


\subsubsection{Proof of Theorem \ref{th:main}} \label{main_result}
We are ready to prove the main theorem of the section: 
\begin{proof}
Following the procedure of last section, give an orientation to the braid and label each crossing following Fig.~\ref{fig:crossings}. After this procedure, if any of the strands oriented upwards is entirely over or entirely under the rest of the braid, following Lemma \ref{prop:canon_plat} it can be considered as one of the closures in the standard closure, putting it aside.

We now perform the \(S\)-move on every \(b,c,d\)-type crossing, obtaining only \(a\)-type crossings. If we have a total of \(n_a, n_b, n_c\) and \(n_d\) crossings of \(a,b,c,d\)-type, then performing all the \(S\)-moves on the crossings of \(b,c,d\)-type, we add a total of \(n_b + n_c + 2\cdot n_d\) new strands. If there are trivial cases, where there are no crossings in two subsequent columns, then move the free strand above the braid. 
Now all the new strands and the ones coming from the trivial cases lay entirely above or underneath the braid. 
Using the result from Lemma \ref{prop:canon_plat}, we have as a result that the link represented by the newly generated braid with classical closure is equivalent to $L$. 

The operations made are all substitutions on the crossings and are done at most once every crossing. When we do the substitution we need to check all the other crossings on the same column in order to write correctly the \(\delta_i\)s of equations (\ref{eqn_C_a}), (\ref{eqn_C_b}), (\ref{eqn_C_c}), (\ref{eqn_C_d}). Suppose we have \(c_i\) crossings in the column \(i\), with \(N = \sum_{i = 1}^n c_i\). The total number of controls needed in column \(i\) are \(c_i \times (c_i-1)\), which gives a total of \(\sum_{i = 1}^n (c_i^2 - c_i)\) controls for the entire braid. 
In the worst case scenario we have only one column containing all the crossings, so the the number of controls needed is \(N \times (N-1) = N^2 - N\). We can deduce that the algorithm is quadratic in the number of the crossings. 
\end{proof}

\begin{example}\rm
In Fig.~\ref{fig:example_plat_standard_R3} there is depicted the passing from a braid representing a link via plat closure to a braid representing the same link but via standard closure. First we give an orientation to the link, then we apply all the \(S\)-moves. Thereafter we remove all the redundant kinks and obtain the final braid, which represents the same link class but via standard closure.
\end{example}

\begin{figure}[h!]
    \centering
    \includegraphics[width = \textwidth]{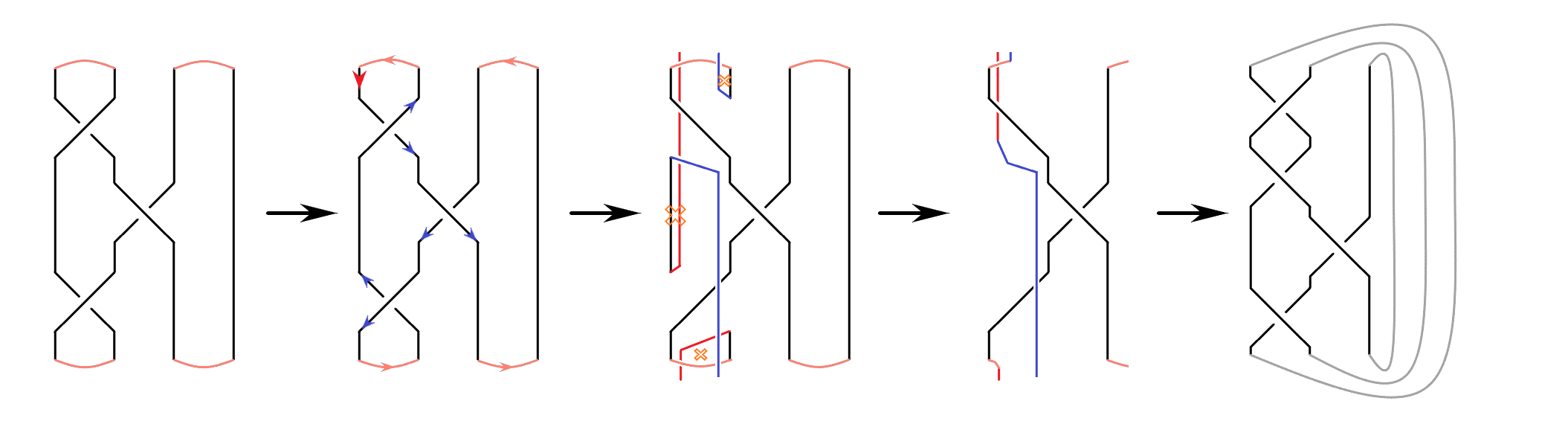}
    \caption{An example of passing from a braid in plat closure to one in standard closure. }
    \label{fig:example_plat_standard_R3}
\end{figure}

\section{The handlebody case} \label{Handlebody}

Let \(H_g\) be the abstract handlebody of genus \(g\). \(H_g\) can be constructed from a 3-ball with $g$ 1-handles attached. Equivalently, it can be constructed as the  product space of the $g$-punctured disc, $D_g$, cross the interval. See Fig.~\ref{fig:standard_handles}.  Namely,
 \(H_g = D_g \times [0,1]\).
 
\begin{figure}[h!]
    \centering
    \includegraphics[width = .9\textwidth]{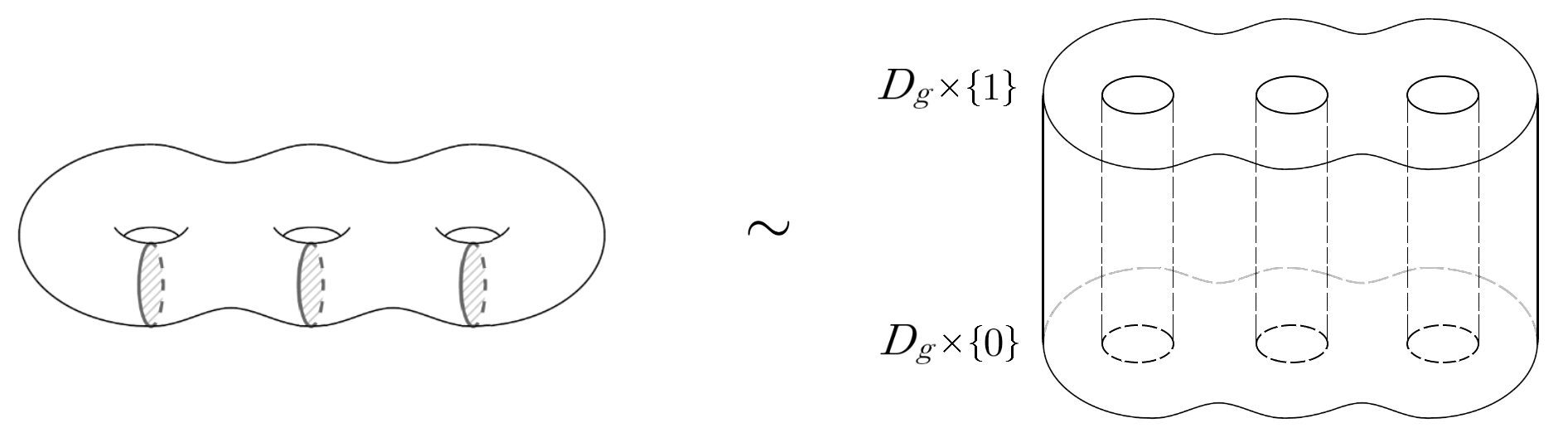}
    \caption{Abstract illustrations  of a handlebody of genus 3. }
    \label{fig:standard_handles}
\end{figure}
    
\noindent Hence, the `decompactified' \(H_g\) may be viewed as embedded in  \(\mathbb{R}^3\) and as the complement of the identity braid \(I_g\) on \(g\) indefinitely extended strands. For viewing \(H_g\) in \(S^3\) we consider all strands of \(I_g\)  meeting at the point at infinity, see Fig.~\ref{fig:mixed_tangle}(a). Then a thickening of the resulting graph is the complementary handlebody of \(H_g\) in \(S^3\). 

Let now \(L\) be a link in \(H_g\). Then, using the above representation and fixing \(I_g\) pointwise, \(L\) may be represented unambiguously  by the mixed tangle (or {\it mixed link}) \(I_g  \cup L\), see Fig.~\ref{fig:mixed_tangle}(b). $I_g$ forms the so-called {\it fixed part} or {\it fixed subbraid} of the mixed link, while $L$ forms the so-called {\it moving part} of the mixed link (see \cite{lambropoulou1997markov,haring2002knot}).

\subsection{Mixed braid groups}\label{section:mixed_braid}

It is shown in \cite{haring2002knot} that  a mixed link may be isotoped to the standard closure of a mixed braid, providing an analogue of the Alexander theorem for knots and links in $H_g$. Furthermore, a corresponding mixed braid equivalence is proved.

A {\it mixed braid} on $n$ moving strands, denoted $I_g \cup B$, is an element of the Artin braid group \(\mathcal{B}_{g+n}\) which contains the identity braid on \(g\) strands,  $I_g$, as a fixed subbraid (see abstraction in Fig.~\ref{fig:mixed_tangle}(c)). By removing $I_g$  we are left with the {\it moving subbraid} $B$ on $n$ strands. In the picture of $H_g = D_g \times [0,1]$, the top $n$ endpoints of the braid $B$ in $H_g$ lie in $D_g \times \{0\}$ and the corresponding bottom \(n\) endpoints of $B$ lie in $D_g \times \{1\}$. 

\begin{figure}[h!]
    \centering
    \includegraphics[width = .9\textwidth]{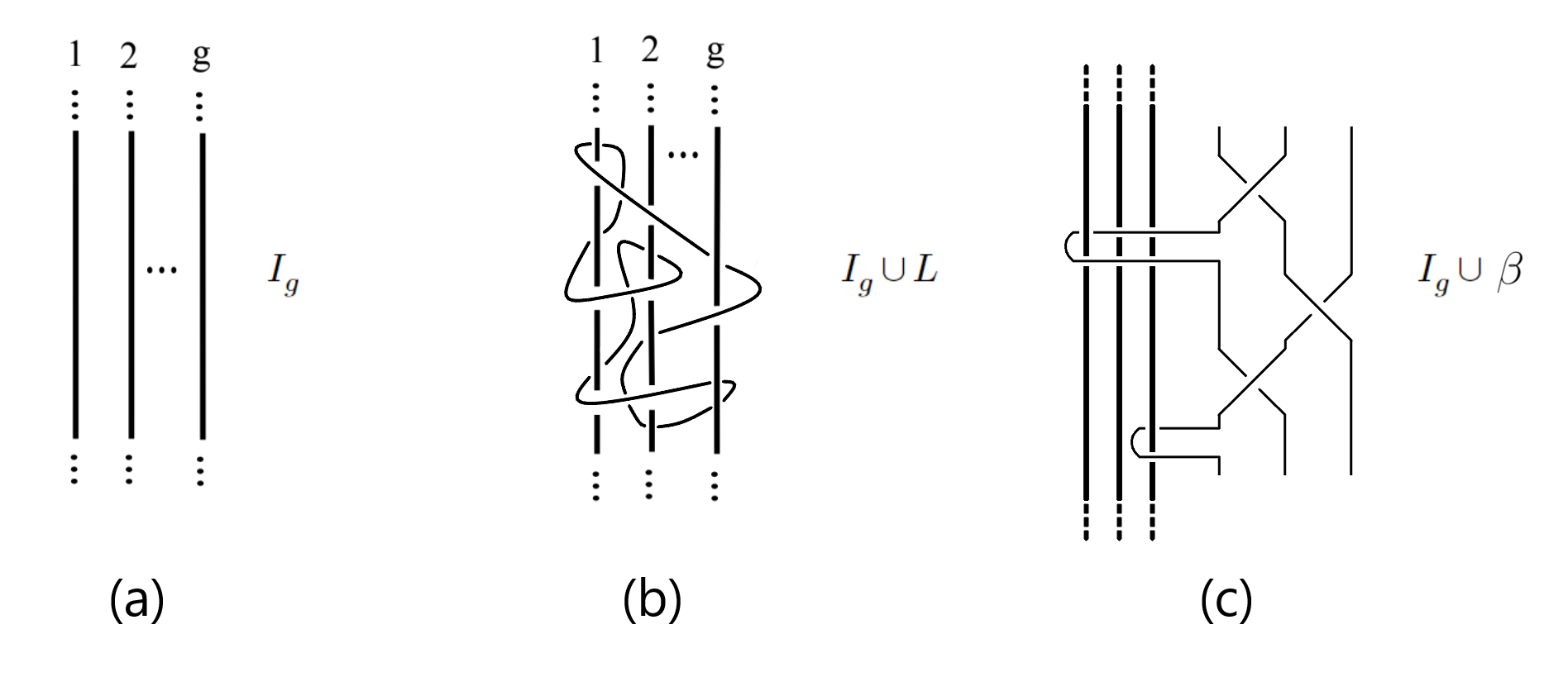}
    \caption{(a) representing \(H_g\) in \(S^3\); (b) a mixed link; 
    (C) a mixed braid.}
    \label{fig:mixed_tangle}
\end{figure}

Furthermore, it is established in \cite{haring2002knot} that the \emph{mixed braid groups}  \(\mathcal{B}_{g,n}\) \cite{lambropoulou2000braid}, $n\in \mathbb{N}$, comprise the algebraic counterpart of knots and links in \(H_g\). \(\mathcal{B}_{g,n}\) is the subgroup of all elements of \(\mathcal{B}_{g+n}\) for which, by removing the last \(n\) strands  we are left with the identity braid  $I_g$. Of course, if we take \(g = 0\), we obtain the standard definition of the Artin braid group in \(\mathbb{R}^3\), \(\mathcal{B}_n\). Elements of \(\mathcal{B}_{g,n}\) are called {\it  mixed braids}. 

In \cite{lambropoulou2000braid}, it is proved that \(\mathcal{B}_{g,n}\) has a presentation with generators: the elementary crossings among moving strands  $\sigma_1, \dots, \sigma_{n-1}$ (see Fig.~\ref{fig:generat_L}) and  the {\it loop generators } $\alpha_1, \dots, \alpha_g$, where \(\alpha_i\) represents a right-handed looping of the first moving strand around the $i$-th fixed strand, as depicted  in Fig.~\ref{fig:generat_L}, where also the inverse of a loop generator is depicted. Further, the generators satisfy the relations: 
\begin{equation} 
\begin{array}{rcll}
\sigma_k \sigma_j &=& \sigma_j \sigma_k & |k-j|>1, \\ 
\sigma_k \sigma_{k+1} \sigma_k &=& \sigma_{k+1} \sigma_k \sigma_{k+1} & 1\leq k\leq n-1,  \\
\alpha_r \sigma_k &=& \sigma_k \alpha_r  & k\geq 2, \  1\leq r \leq g, \\
\alpha_r \sigma_1 \alpha_r \sigma_1 &=& \sigma_1 \alpha_r \sigma_1 \alpha_r & 1 \leq r \leq g, \\
\alpha_r (\sigma_1 \alpha_s \sigma_1^{-1}) &=& (\sigma_1 \alpha_s \sigma_1^{-1}) \alpha_r & s < r.
\end{array}
\end{equation}


\begin{figure}[h!]
    \centering
    \includegraphics[width = .8\textwidth]{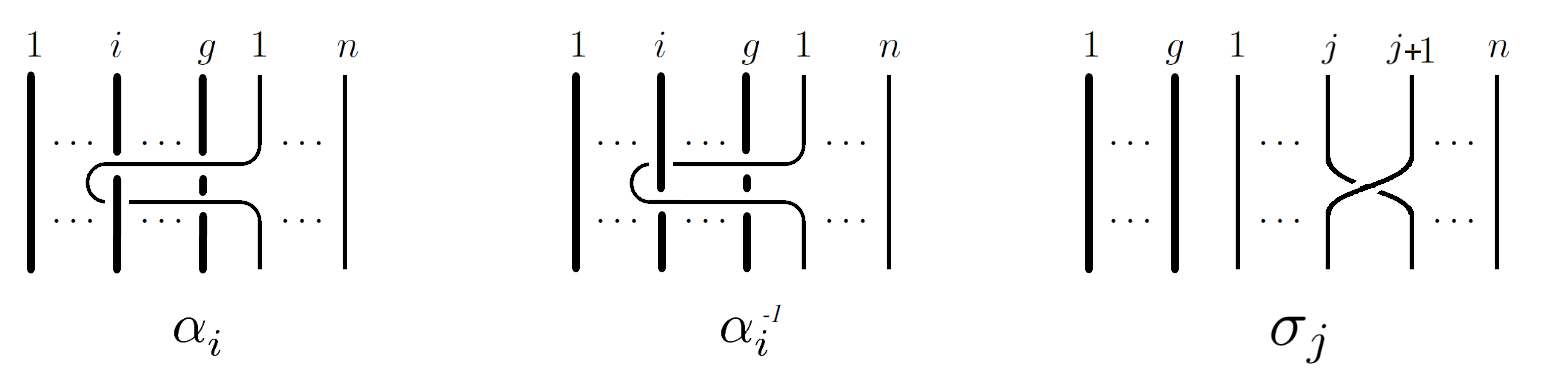}
    \caption{The generators of \(\mathcal{B}_{g,n}\).}
    \label{fig:generat_L}
\end{figure}

\begin{note}\label{note_notation_alpa} \rm 
It should be mentioned that in \cite{lambropoulou2000braid} and consequently in \cite{haring2002knot} and \cite{lambropoulou2006algebraic} the loop generators  \(\alpha_i\) are denoted by \(a_i\). In this paper we prefer to avoid the notation \(a_i\), as this is used in the next section for denoting a type of loop generators of the surface braid group.
\end{note}

\subsection{The standard closure and the plat closure for mixed braids}

\begin{definition}
Let \(\beta \in \mathcal{B}_{g,n}\) be a  mixed braid.  The {\it standard closure} of \(\beta\) is an oriented mixed link defined by joining with simple unlinked arcs each pair of corresponding endpoints of the moving braid strands. See thin grey arcs in Fig.~\ref{fig:mixed_tangle_2}(a). 
\end{definition}

Conversely, an oriented mixed link can be braided to a mixed braid with standard closure \cite{haring2002knot} (an analogue of the classical Alexander theorem for knots and links in $H_g$). Moreover, mixed link isotopy is translated in \cite{haring2002knot} into a mixed braid equivalence  in $\bigcup_n{\mathcal{B}_{g,n}}$ (an analogue of the classical Markov theorem). For details the reader may consult \cite{haring2002knot}. 

\begin{figure}[h!]
    \centering
    \includegraphics[width = .85\textwidth]{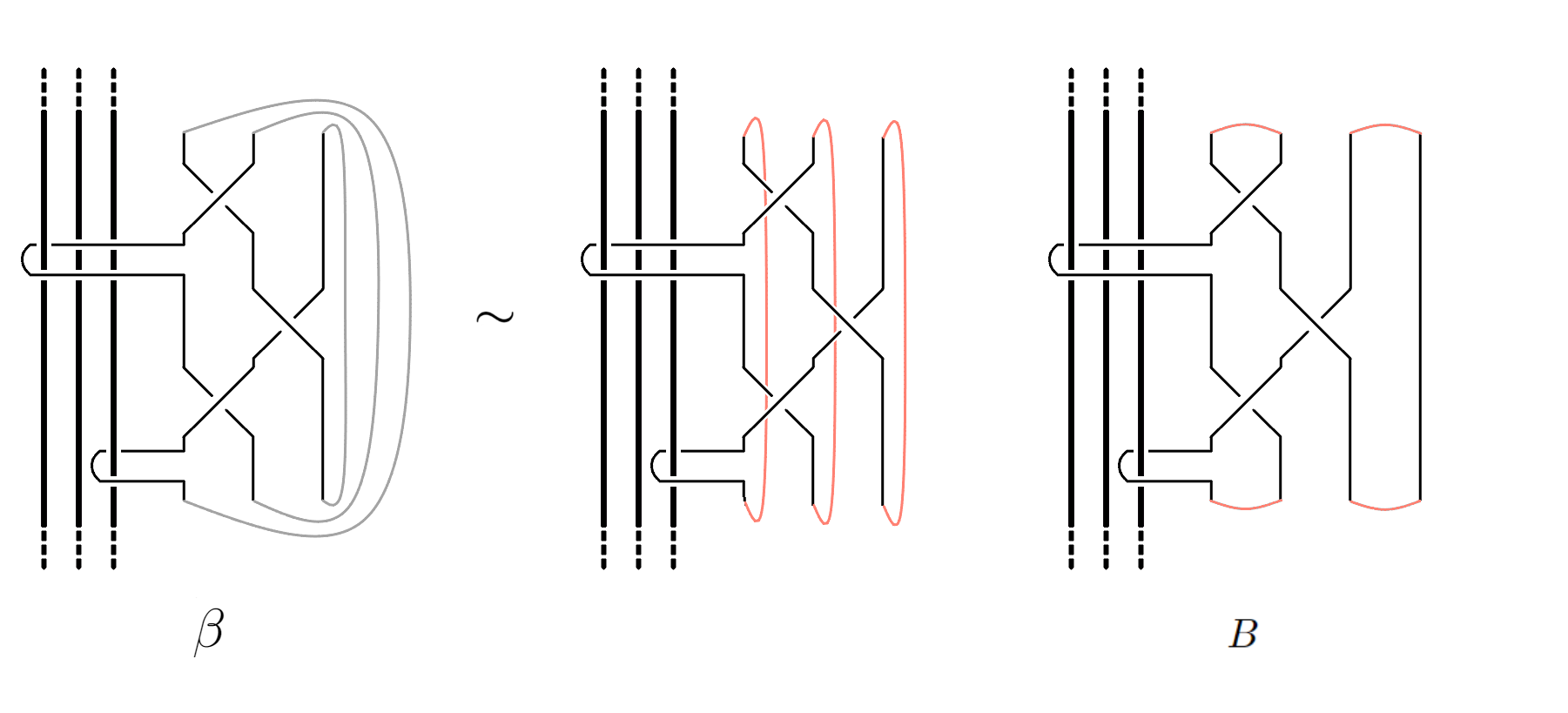}
    \caption{A mixed braid: with standard closure ($\beta$); with plat closure ($B$).}
    \label{fig:mixed_tangle_2}
\end{figure}

\begin{definition} \label{mixedplat}
Let $B \in \mathcal{B}_{g,2n}$ be a mixed braid with an even number of moving strands. The {\it plat closure} of $B$ is a mixed link defined as follows:  we first number the set of top (resp. bottom) endpoints from left to right, with numbers $1, 2, \ldots, 2n$ (resp. $1', 2', \dots, 2n'$). To each pair $(2i-1, 2i)$ (resp. $({2i-1}', 2i')$) we use a simple arc, say $\gamma_i$ (resp.$\gamma_i'$), as joining arc for the pair of endpoints. See thin grey arcs in Fig.~\ref{fig:mixed_tangle_2}(b). The resulting unoriented mixed link shall be referred to as {\it mixed plat}. 
\end{definition}

Note that the notion of plat closure in the definition above is well-defined, by recalling that $B $ is an element in $\mathcal{B}_{g+2n}$ and adapting the definition of plat closure for a classical braid. Note further that Definition~\ref{mixedplat} could be extended also in the case of an {\it odd } number of moving strands, using the same rationale as in $S^3$; recall discussion after Definition~\ref{odd_extension_R3}. 

\smallbreak
Although it is shown in \cite{haring2002knot} that any link \( L \) in the handlebody \(H_g\) can be represented by a mixed braid with isotopic standard closure, for the plat closure case we have to adapt the braiding algorithm of \cite{birman1976stable, cattabriga2018markov} to the handlebody setting. As it turns out, \(L\) can be  braided to a mixed braid \(B\) in some \(\mathcal{B}_{g,2n}\), representing the link via plat closure. This is presented in a sequel paper to this one (as well as the corresponding plat equivalence)  \cite{handlebody}. We will also prove the braiding here as a corollary, after giving a transformation from standard  to plat closure.

\subsection{From standard closure to plat closure}

Let now \(B \in \mathcal{B}_{g,2n}\) and \( \beta\in \mathcal{B}_{g,m}\) be two mixed braids representing a link \(L\) in $H_g$ via plat closure and standard closure respectively. The question is how these mixed braids are related. 

The first part of the result in the setting of  $H_g$  is obtained in the same way as in the case of \(\mathbb{R}^3\): 

\begin{theorem}\label{th:stand_plat_h}
Let \(H_g\) be the handlebody of genus \(g\). Given a link \(L\) in \(H_g\) represented by a mixed braid \(\beta\in \mathcal{B}_{g, m}\) with standard closure, it is possible algorithmically to generate a mixed braid \(B \in \mathcal{B}_{g,2m}\) which represents an equivalent link but with plat closure. The algorithm has a computational complexity of \(O(M)\), where \(M\) is the number of crossings present in the braid \(\beta\). 
\end{theorem}

\begin{proof}
Let  \(L\) be a link in \(H_g\).  Recall that by the results in \cite{haring2002knot}, there exist mixed braids  representing \(L\) via the standard closure. 
Recall also that the mixed braid group \(\mathcal{B}_{g,n}\) is generated by the elementary crossings \(\sigma_j, \ j \in \{1, \dots, n-1\}\) and the loop generators \(\alpha_i,  \ i\in\{1, \dots, g\}\), see Fig.~\ref{fig:generat_L}. For this proof we shall use the techniques employed for proving Theorem~\ref{alternativeplatbraiding} for the case of $\mathbb{R}^3$, and we will expand them for taking care of the loop generators. More precisely, let \( \beta\in \mathcal{B}_{g,m}\) be a mixed braid representing \(L\) via the standard closure. View abstraction in rightmost illustration of Fig.~\ref{fig:mixed_tangle_2}.  We follow the same construction of Section \ref{from_standard}: we take the \(i\)-th closing arc and we let it  pass isotopically underneath the mixed braid till it comes to  the right of the \(i\)-th strand, becoming the \((2i-1)\)-th and  \(2i\)-th moving strands in the plat representation, recall Fig.~\ref{fig:L-C}. Note that, during this process the fixed part of the mixed braid (representing the holes of the handlebody) does not interfere, since the closing arcs lie on the right of the mixed braid. Therefore, the algebraic transformation for the generators $\sigma_j$ is still equation (\ref{eqn:transformation}). Note that, since the substitutions are done once for each crossing, the computational complexity is linear in the number of crossings.
\end{proof}

\subsection{From plat closure to standard closure}

In this section we will describe an algorithm  that, knowing the representation of \(L\) through a braid closed via plat closure, it gives us a representative of \(L\) in the other closure. The operation for $H_g$ is quite similar to the one described in the previous sections for the case of $\mathbb{R}^3$. We only need to, additionally, take care of the loop generators \(\alpha_i\) of the mixed braid. We now state our main result for this section. 

\begin{theorem}\label{teo:handleb_plat_standard}
Let \(H_g\) be a handlebody of genus \(g\). Given a link \(L\) in \(H_g\) represented by a mixed braid \(B\in \mathcal{B}_{g, 2n}\) with plat closure, it is possible algorithmically to generate a mixed braid \(\beta\in \mathcal{B}_{g,m}\) which represents an equivalent link but with standard closure. The algorithm has a computational complexity of \(O(N^2)\) where \(N\) is the number of crossings present in the braid \(B\). 
\end{theorem}

\begin{proof}
This way round, from plat closure to standard closure, needs more attention than the previous case. Let \(B\in \mathcal{B}_{g, 2n}\) be a mixed braid representing \(L\) via plat closure. View abstraction in rightmost illustration of Fig.~\ref{fig:mixed_tangle_2}. After giving an orientation to the moving part of our mixed braid, the situation is very similar to the one that we have already described in Section \ref{construction}, but with one difference: now we can  also have loop generators, oriented upward or downward. 

As for the case of \(\mathbb{R}^3\), we identify the space between moving strand \(i\) and moving strand \(i+1\) as the \(i\)-th column. We then call ``0''-th column the one defined between the \(g\)-th fixed strand and the first moving strand of the mixed braid. Now we are ready to apply our transformation for turning \(B\) into an algebraic mixed braid \(\beta\) with equivalent closure.

We first deal with the crossings between moving strands. These are treated exactly as in Theorem \ref{th:main}, since the moving strands are all gathered in the right-hand-side of the mixed braid. Furthermore, as their numbering is the same as for classical braids, formulae (\ref{eqn_C_a}), (\ref{eqn_C_b}), (\ref{eqn_C_c}) and (\ref{eqn_C_d}) apply directly on the columns $1, 2, \ldots, 2n-1$ of the word \(B\in \mathcal{B}_{g, 2n}\). The only point of concern is the first column, from which loop generators may be reaching out to the fixed strands. Yet, according to our algorithm, the new braid strands (red and blue) created by the applications of \(S\)-moves, lie entirely within the columns where the   \(S\)-moves are applied. This means that the loop generators will not be affected by the above formulae. After dealing with all moving crossings of types \(b, c\) and \(d\), we are left with only type-\(a\)  crossings.

\begin{figure}[h!]
    \centering
    \includegraphics[width=.72\textwidth]{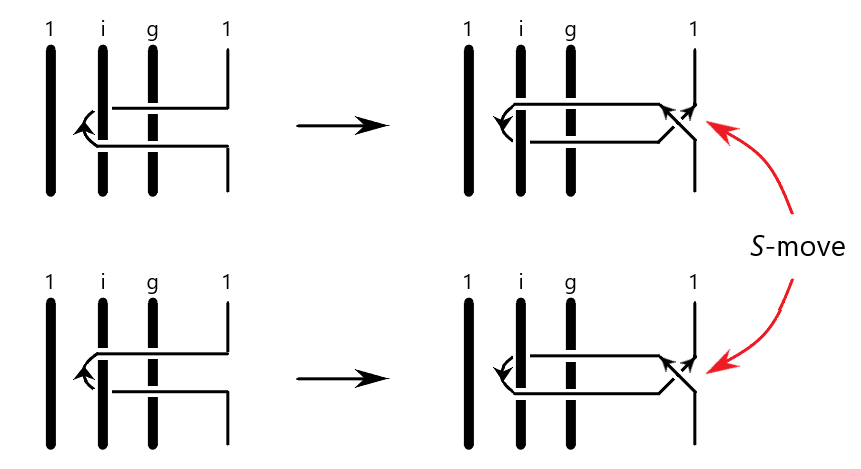}
    \caption{Twisting an \(\alpha_i^k\) oriented upward changes \(k\) and generates a \(d\)-type crossing.}
    \label{fig:twist_a}
\end{figure}

We shall next deal with the loop generators $\alpha_1, \ldots, \alpha_g$ and their inverses. In the final mixed braid with standard closure that we will obtain, all loop generators and their inverses need to be oriented downward. So, if a loop generator (or its inverse) is oriented downward it remains unchanged. If, however, a loop generator or its inverse is oriented upward, we perform a twist following the isotopy depicted in Fig.~\ref{fig:twist_a}, causing a change of its orientation and its type (from unoriented positive to inverse and vice-versa) and creating at the same time  a \(d\)-type moving crossing in the 0-th column. 

The next thing to do is to perform an \(S\)-move for every \(d\)-type moving crossing. Assume we have  \(k\) loop twists and, thus, \(S\)-moves in total. When we focus on the result in the 0-th column, we have (see Fig.~\ref{fig:0_column_S}):
\begin{itemize}
    \item any non-twisted \(\alpha_i^{\pm 1}\) becomes \ \(\sigma_{2k} \sigma_{2k-1} \dots \sigma_{k+1} \sigma_{k}^{-1} \sigma_{k-1}^{-1} \dots \sigma_1^{-1} \alpha_i^{\pm 1} \sigma_1 \sigma_2 \dots \sigma_{k} \sigma_{k+1}^{-1} \sigma_{k+2}^{-1} \dots \sigma_{2k}^{-1}\)
    \item any twisted  \(\alpha_i^{\pm 1}\) becomes \  \(\alpha_i^{\mp 1} \sigma_1 \sigma_2 \dots \sigma_{k-1} \sigma_{2k} \sigma_{2k-1} \dots \sigma_{k+1} \sigma_k\)
\end{itemize}
Of course, if we have two \(S\)-moves one after the other, without any \(\sigma_1\) in the middle, we will have to perform a Reidemeister~I move in order to bring the braid into standard position, as shown in Fig.~\ref{fig:0_column_S}. 

The last thing to do is the top part and bottom part management for the plat closure arcs. Again, as these do not interfere with the fixed strands, they are dealt exactly as described in Section~\ref{top_bottom_management} for the case of \(\mathbb{R}^3\).

Note finally that all the substitutions made on the \(\alpha_i\) generators are done in linear time, because all the \(S\)-moves done in the first column are of the same type. Then, following the line of reasoning of \ref{main_result}, the computational complexity remains quadratic in the number of the \(\sigma_i\) generators.

\begin{figure}[h!]
    \centering
    \includegraphics[width=.6\textwidth]{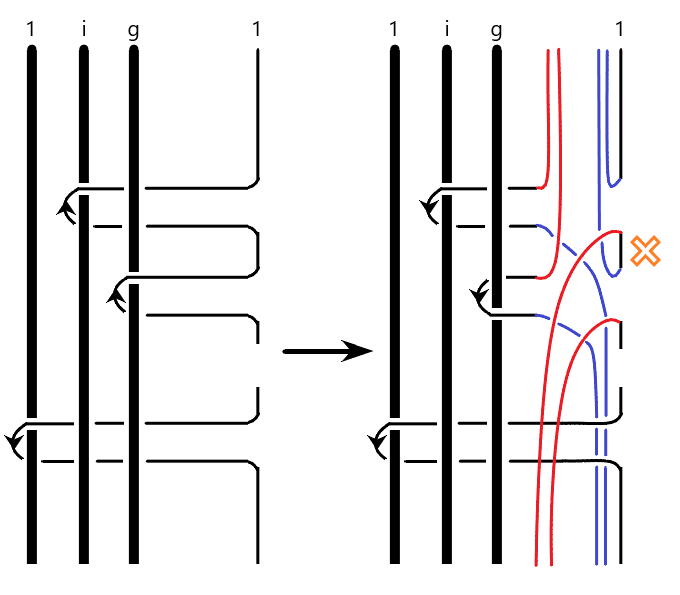}
    \caption{The result after all the twists and \(S\)-moves on the 0-th column. }
    \label{fig:0_column_S}
\end{figure}

\end{proof}

\begin{example} \rm 
In Fig.~\ref{fig:example_plat_standard_hand} we illustrate an example for passing from a mixed braid representing a link via plat closure to a mixed braid representing the same link but via standard closure in the handlebody setting. First we give an orientation to the link, then we apply all the \(S\)-moves. Thereafter we remove all the extra crossings and twist the loop generators oriented upward. We apply the last \(S\)-moves in the 0-th column and remove again all the extra crossings, obtaining the final braid which represents the same link class but via standard closure.
\end{example}

\begin{figure}[h!]
    \centering
    \includegraphics[width = \textwidth]{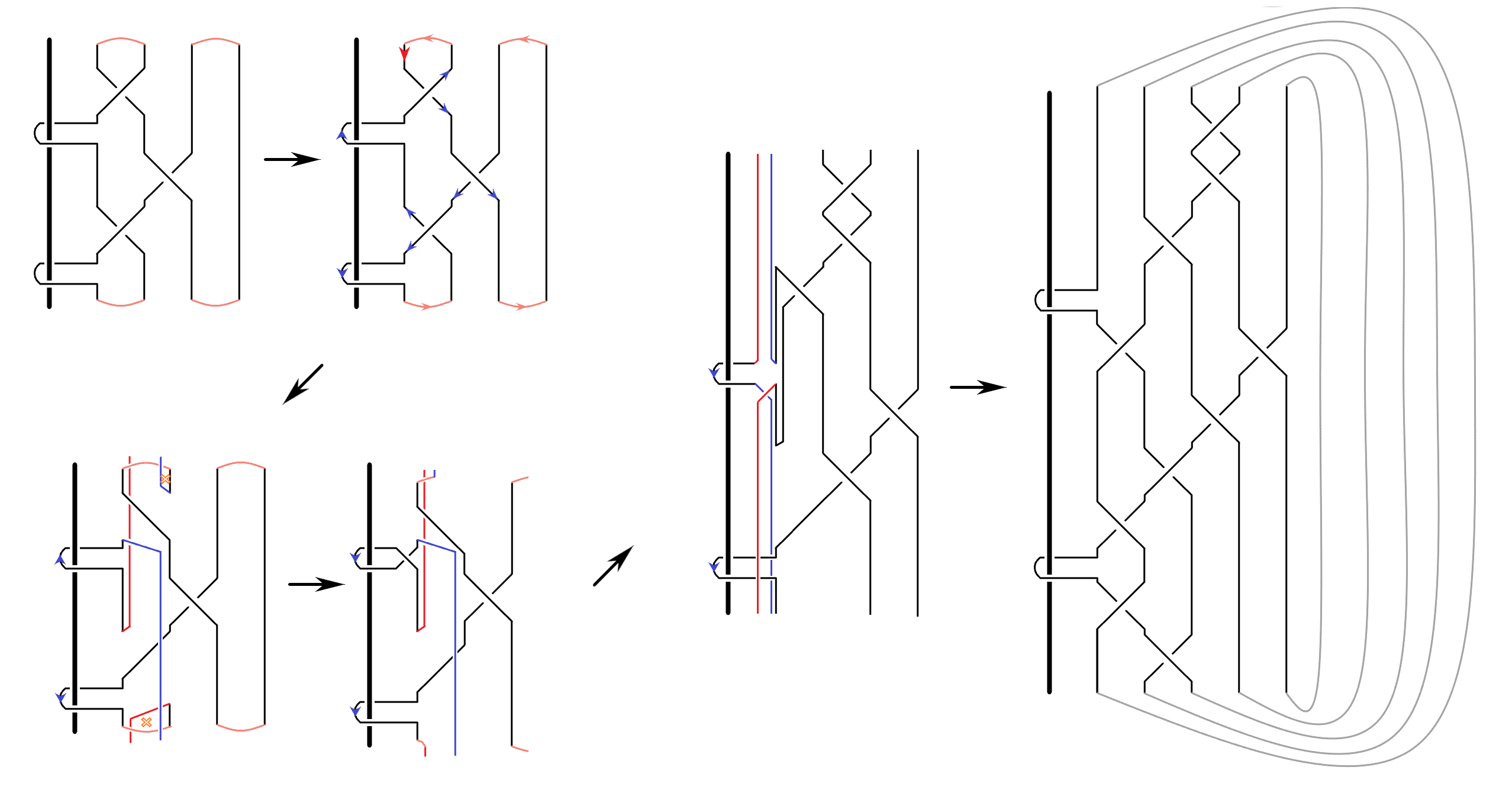}
    \caption{An example of passing from a mixed  braid in plat closure to one in standard closure in the handlebody setting.  }
    \label{fig:example_plat_standard_hand}
\end{figure}

As a corollary of  Theorem~\ref{teo:handleb_plat_standard} and by results in \cite{haring2002knot} (recall Section~\ref{section:mixed_braid}), we have proven the analogue of the Alexander theorem for a link in a handlebody being isotopic to the plat closure of a mixed braid: 
\begin{theorem}
Let \(H_g\) be a genus \(g\) handlebody. Every  link \(L\) in \(H_g\) may be braided to a mixed braid in \(\mathcal{B}_{g,n}\), the plat closure of which is isotopic to \(L\). 
\end{theorem}


\begin{remark} \rm
The loop generators of the mixed braid group lead to a further consideration. Namely, a (classical or mixed) pure braid can be expressed in terms of loop generators (\cite{artin1947theory},\cite{lambropoulou2000braid}) and this can be extended to any (classical or mixed) braid, namely it can be expressed as a composition of a pure braid and a standard permutation braid (see \cite{chow48pure},\cite{lambropoulou2000braid}). So, both in the classical case and the handlebody case, one could use in the algorithm for passing from plat to standard closure the technique of twisting the opposite oriented pure braid loop generators (instead of classical crossings) and handling a controlled number of classical crossings. 
We would like to thank Dimitrios Loulas for providing the idea of using pure braid generators in place of crossings. Even though the number of extra strands that we create during this different algorithm can be slightly optimized than in our standard case, the computational complexity of the algorithm is not improved. For this reason, it could be useful to use this remark when looking for a lower bound for the creation of extra strands while passing from one closure to the other. The above considerations apply also to the thickened surface case discussed next.

\end{remark}

\section{The thickened surface case} \label{thickened}

\subsection{Surface braid groups}

Let  \(\Sigma_g\) be a c.c.o. genus \(g\) surface and let \(\Sigma_g \times [0,1]\) be a thickening of \(\Sigma_g\), view Fig.~\ref{fig:thickened_surface}.

\begin{figure}[h!]
    \centering
    \includegraphics[width = .52\textwidth]{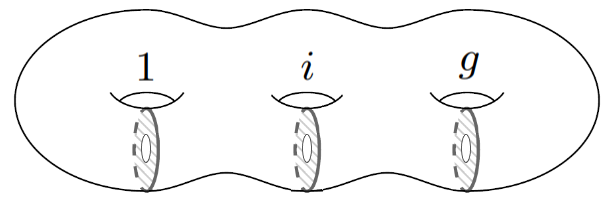}
    \caption{Abstract illustration  of a thickened surface of genus 3. }
    \label{fig:thickened_surface}
\end{figure}

Consider further the set of points \(\mathcal{P} = \{P_1, \dots, P_n\}\)  in \(\Sigma_g\). Following \cite{bellingeri2004presentations}, the \emph{surface braid group} \(\mathcal{B}_{\Sigma_g, n}\) is defined as the fundamental group of the configuration space of \(\mathcal{P}\) in \(\Sigma_g\). This is the group with elements  \(n\)-tuples \(\Psi = (\psi_1, \dots, \psi_n)\) of paths \(\psi_i : [0,1] \rightarrow \Sigma_g\) such that \(\psi_i(0) = P_i, \ \psi_i(1) \in \mathcal{P}, \forall i \in \{1, \dots, n\}\) and \(\psi_1(t), \dots, \psi_n(t)\) are distinct points of \(\Sigma_g, \forall t \in [0,1]\). It is important to note that, by the definition of the surface braid group \(\mathcal{B}_{\Sigma_g, n}\), it is always possible to consider the points in \(\mathcal{P}\) lying in a small disc \(A \subset \Sigma_g\) away from the genus holes (see \cite{bellingeri2004presentations},\cite{cattabriga2018markov}), as illustrated in Fig.~\ref{fig:thick_points}. 

\begin{figure}[h!]
    \centering
    \includegraphics[width = .7\textwidth]{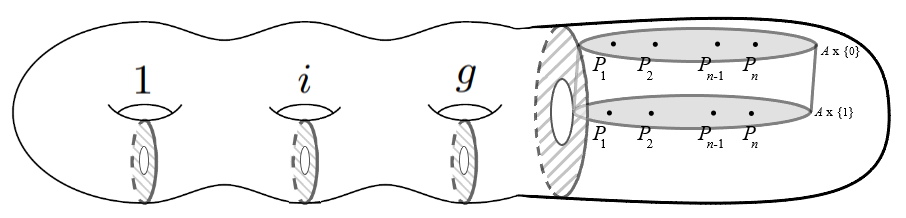}
    \caption{The set of points \(P_1, \dots, P_n\) on the two levels \(\Sigma_g \times \{0\}\) and \(\Sigma_g \times \{1\}\).}
    \label{fig:thick_points}
\end{figure}

In \cite{bellingeri2004presentations} Bellingeri obtains a presentation for the surface braid group \(\mathcal{B}_{\Sigma_g, n}\), as follows: the generators of the group are: 
\begin{itemize}
    \item crossings \(\sigma_j, j = 1, \dots, n-1\), which exchanges points \(P_{j}\) and \(P_{j+1}\), and keeps fixed all other points;   
    \item loop generators \(a_i, i = 1, \dots, g\), which keep fixed all points but \(P_1\) and loop around the longitude of the \(i\)-th genus  hole;  
    \item loop generators \(b_i, i = 1, \dots, g\), which keep fixed all points but \(P_1\) and loop  around the meridian of the \(i\)-th handle.  
\end{itemize}
An illustration of the generators of \(\mathcal{B}_{\Sigma_g, n}\) as elements of the fundamental group of the configuration space of \(\mathcal{P}\) in \(\Sigma_g\), is given in Fig.~\ref{fig:generators_surface}, where we have dropped the notation $P_i$ for the $i$th point. 

\begin{figure}[h!]
    \centering
    \includegraphics[width = .8\textwidth]{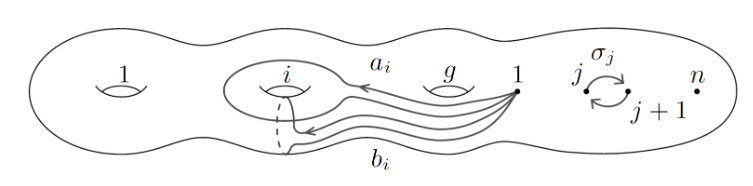}
    \caption{The generators of \(\mathcal{B}_{\Sigma_g,n}\) represented on the surface \(\Sigma_g\). }
    \label{fig:generators_surface}
\end{figure}

A complete set of relations for \(\mathcal{B}_{\Sigma_g,n}\) is the following: 
\begin{itemize}
    \item Braid relations: 
        \begin{align*}
            \sigma_j \sigma_{j+1} \sigma_j & = \sigma_{j+1} \sigma_{j} \sigma_{j+1}, \qquad & j & = 1, \dots, 2n-2 \\
            \sigma_j \sigma_k  & = \sigma_k \sigma_j, \qquad & j,k &= 1, \dots, 2n-1, \>\> |j-k| \geq 2 \\
        \end{align*}
    \item Mixed relations: 
    \begin{align*}
        a_r \sigma_j & = \sigma_j a_r, \qquad & 1 &\leq r \leq g, j\neq 1 \\
        b_r \sigma_j & = \sigma_j b_r, \qquad & 1 &\leq r \leq g, j\neq 1 \\
        \sigma_1^{-1} a_r \sigma_1^{-1} a_r & = a_r \sigma_1^{-1} a_r \sigma_1^{-1}, \qquad & 1 &\leq r\leq g \\
        \sigma_1^{-1} b_r \sigma_1^{-1} b_r & = b_r \sigma_1^{-1} b_r \sigma_1^{-1}, \qquad & 1 &\leq r\leq g \\
        \sigma_1^{-1} a_s \sigma_1 a_r & = a_r \sigma_1^{-1} a_s \sigma_1, \qquad & s &< r  \\
        \sigma_1^{-1} b_s \sigma_1 b_r & = b_r \sigma_1^{-1} b_s \sigma_1, \qquad & s &< r \\
        \sigma_1^{-1} a_s \sigma_1 b_r & = b_r \sigma_1^{-1} a_s \sigma_1, \qquad & s &< r \\ 
        \sigma_1^{-1} b_s \sigma_1 a_r & = a_r \sigma_1^{-1} b_s \sigma_1, \qquad & s &< r \\ 
        \sigma_1^{-1} a_r \sigma_1^{-1} b_r & = b_r \sigma_1^{-1} a_r \sigma_1, \qquad & s &< r \\ 
        [a_1, b_1^{-1}] \dots [a_g, b_g^{-1}] & = \sigma_1 \sigma_2 \dots \sigma_{2n-1}^2 \dots \sigma_2 \sigma_1  \\
    \end{align*} 
\end{itemize} 
where \([a,b^{-1}] = ab^{-1}a^{-1}b\).

\begin{note}\label{note:cylinder_A} \rm
For geometric interpretations of surface braids we consider the lifting of a surface braid \(B\), such that all the \(\sigma_j\) generators are lying  in the cylinder \(A \times I\), as illustrated in Fig.~\ref{fig:thick_points}. In this setting, the crossing generators \(\sigma_j\) are lifted as depicted in Fig.~\ref{fig:generat_L}. The loop generators \(a_i\) are  lifted as depicted in the right-hand side of  Fig.~\ref{fig:catta_lambro_a_i} (see \cite{lambropoulou2000braid}). In the  left-hand side of the figure, the loop generators \(\alpha_i\) of the mixed braid group \(\mathcal{B}_{g,n}\) are illustrated (recall Note \ref{note_notation_alpa}). Finally, the loop generators \(b_i\)  around the handles and their inverses are  represented in \cite{cattabriga2018markov} by abstract arrows, as illustrated in Fig.~\ref{fig:generat_B}. 

The above geometric interpretations allow us to further represent in $\mathbb{R}^3$ surface braids in \(\mathcal{B}_{\Sigma_g, n}\) as augmented classical braids in \(\mathcal{B}_{g+n}\) with the first \(g\) strands fixed, forming the identity braid $I_g$, and the last \(n\) strands moving.  Namely as {\it augmented mixed braids} in the mixed braid group \(\mathcal{B}_{g,n}\) \cite{cattabriga2018markov}. So, a surface braid can be represented as a mixed braid augmented by abstract arrows. As for \(\mathcal{B}_{g,n}\) and the case of handlebody, the fixed subbraid $I_g$ represents topologically the genus holes of \(\Sigma_g\).  An example of a surface braid in this setting is depicted in Fig.~\ref{fig:surface_braid_example}.
\end{note}

\begin{figure}[h!]
    \centering
    \includegraphics[width = .65\textwidth]{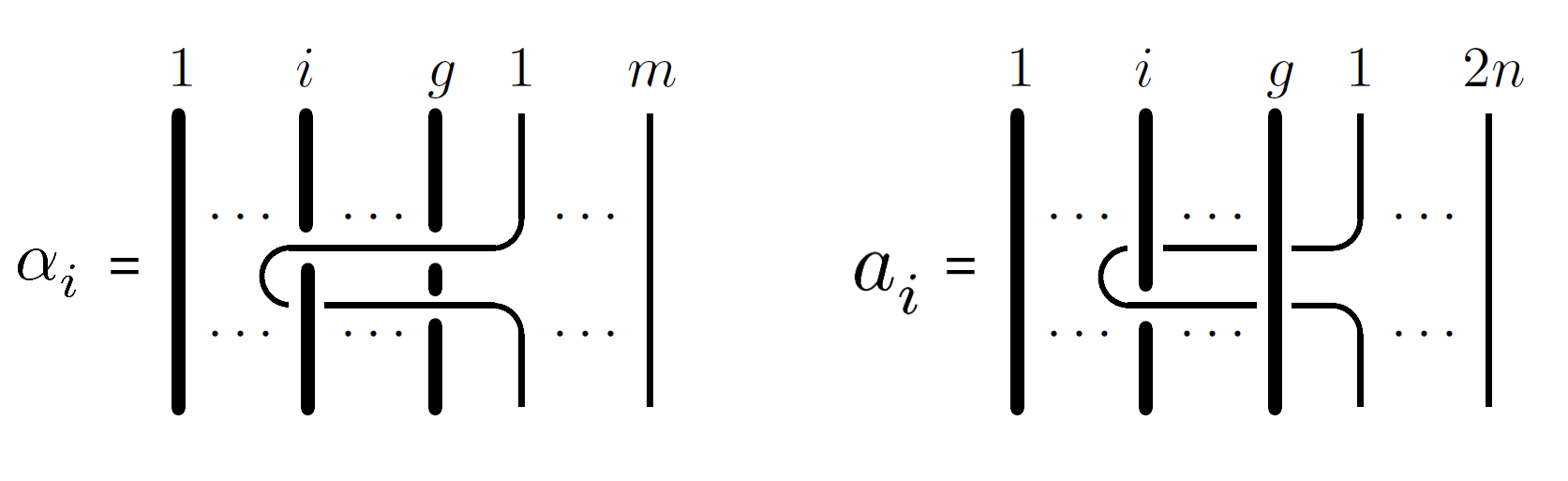}
    \caption{Two different loop generators around the genus: on the left, the ones used in \cite{lambropoulou2000braid} for the handlebody, and on the right the ones used in \cite{cattabriga2018markov} for the thickened surface. }
    \label{fig:catta_lambro_a_i}
\end{figure}

\begin{figure}[h!]
    \centering
    \includegraphics[width = .65\textwidth]{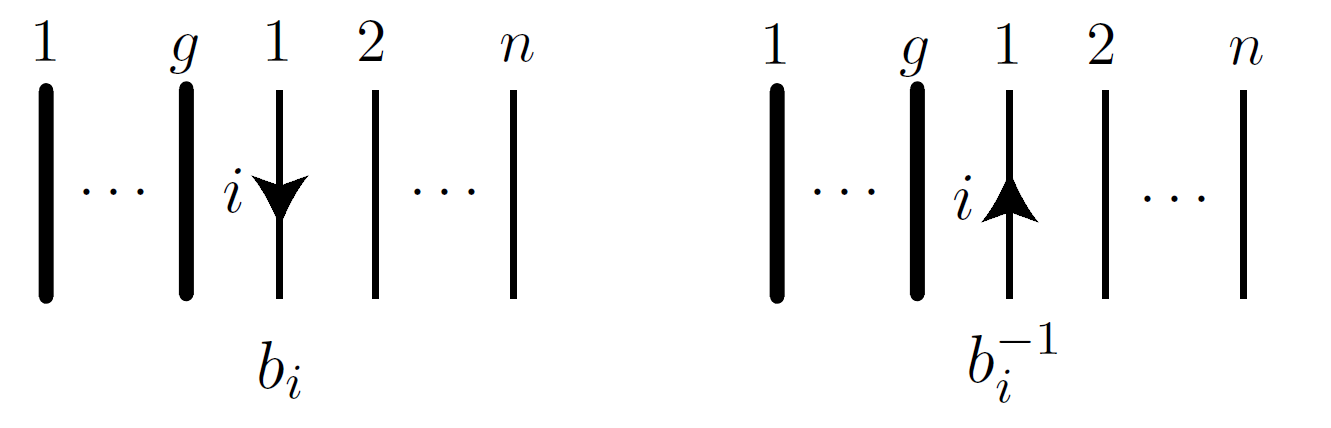}
    \caption{The \(b_i\) generator of \(\mathcal{B}_{\Sigma_g,n}\). Note that the small arrow is added only to clarify the number of the hole and the direction around it.}
    \label{fig:generat_B}
\end{figure}

\begin{figure}[h!]
    \centering
    \includegraphics[width = .4\textwidth]{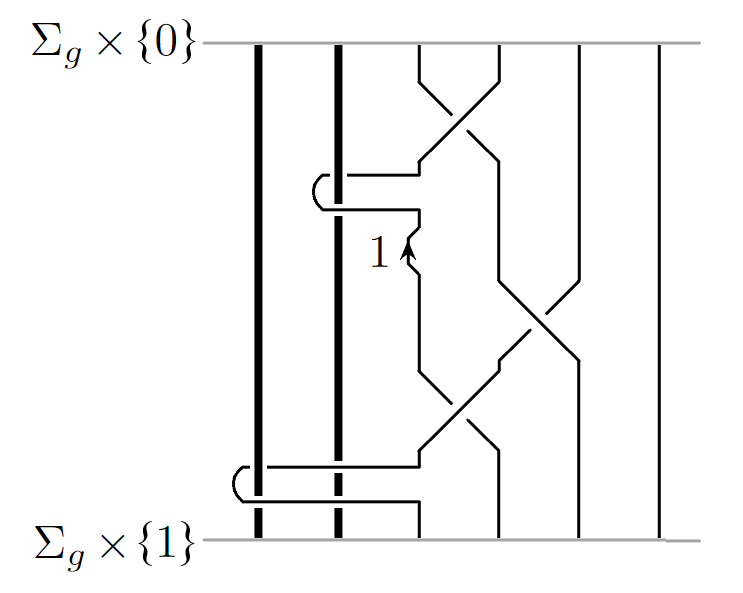}
    \caption{An example of surface braid. }
    \label{fig:surface_braid_example}
\end{figure}

\begin{note} \rm
As we see in Fig.~\ref{fig:catta_lambro_a_i}, the definitions of the  two genus loop generators \(\alpha_i\) in \cite{lambropoulou2000braid} and \(a_i\) in \cite{cattabriga2018markov}  are slightly different. More precisely, they are related via the equation: 
\begin{equation}\label{eqn_alpha_a}
\alpha_i^k = a_g^{-1} a_{g-1}^{-1} \dots a_{i+1}^{-1} a_i^{-k} a_{i+1} \dots a_{g-1} a_g, \qquad k = \pm 1 
\end{equation}
The first step of the conversion from $\alpha_i^k$ to $a_i^{-k}$ in Equation (\ref{eqn_alpha_a})  is depicted in Fig.~\ref{fig:pass_alfa_a}. 
\end{note}

\begin{figure}[h!]
    \centering
    \includegraphics[width=.8\textwidth]{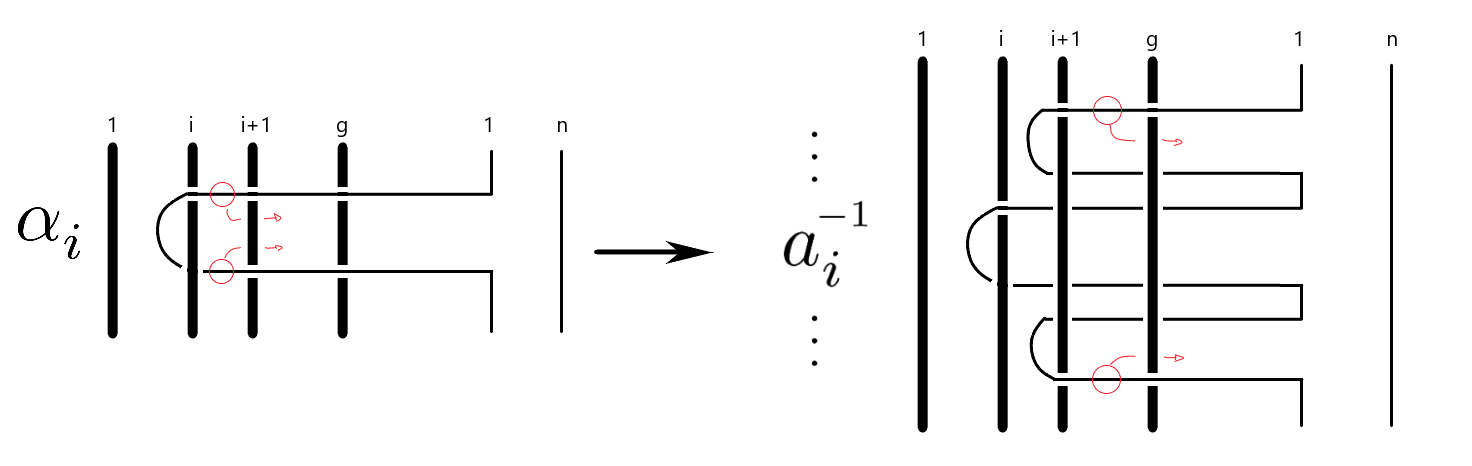}
    \caption{Transformation from \cite{lambropoulou2000braid} to \cite{cattabriga2018markov}. }
    \label{fig:pass_alfa_a}
\end{figure}

\subsection{The plat closure and the standard closure for surface braids} \label{subsec:plat_stand_thick}

\begin{definition}\label{def:plat_closure_thick}
Let \(B\) be an element of some surface braid group \(\mathcal{B}_{\Sigma_g, 2n}\), where \(\mathcal{P}\) has even cardinality \(2n\). {\it The plat closure} of \(B\) is defined, in analogy to the classical case and the handlebody case, by using small simple arcs \(\subset A\) connecting points \(P_{2i-1}\) and \(P_{2i}, \ i= 1, \dots, n\), both at level \(\{0\}\) and \(\{1\}\), see Fig.~\ref{fig:thick_plat_closure_A}. The resulting unoriented surface link shall be referred to as {\it surface plat}.
\end{definition}

Note that, as in the case of mixed braids, the notion of plat closure in the definition above is well-defined, by recalling that $B$ is an augmented element in $\mathcal{B}_{g,2n}$ and thus also in  $\mathcal{B}_{g+2n}$ and adapting the definition of plat closure for a classical braid. Note, further, that Definition~\ref{def:plat_closure_thick} may be extended also in the case of an {\it odd } number of moving strands, using the same rationale as in $S^3$ and in the case of handlebody. Recall discussion after Definition~\ref{mixedplat}.

\begin{figure}[h!]
    \centering
    \includegraphics[width = .7\textwidth]{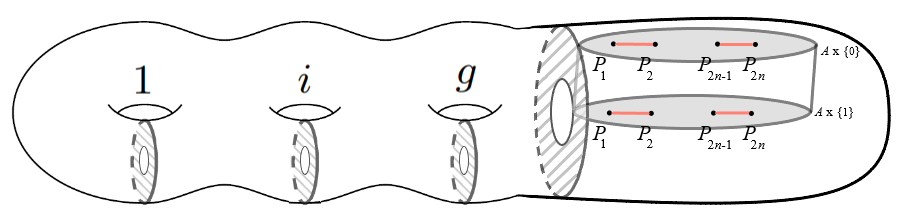}
    \caption{An example of plat closure in our setting. Note that here the number of points is even. }
    \label{fig:thick_plat_closure_A}
\end{figure}


\begin{definition}
Given a braid \(\beta \in \mathcal{B}_{\Sigma_g, m}\), we define {\it the standard closure} of \(\beta\) as the one obtained using simple arcs connecting the corresponding points of \(\mathcal{P}\) at level \(\{0\}\) and at level \(\{1\}\) on the boundary of \(A \times I\), following the red arcs as in Fig.~\ref{fig:standard_closure_thick}.
\end{definition}

Note that, as in the case of mixed braids, the notion of standard closure in the definition above is well-defined, by recalling that $\beta$ is an augmented element in $\mathcal{B}_{g,2n}$ and thus also in  $\mathcal{B}_{g+2n}$ and adapting the definition of standard closure for a classical braid.
 
\begin{figure}[h!]
    \centering
    \includegraphics[width = .7\textwidth]{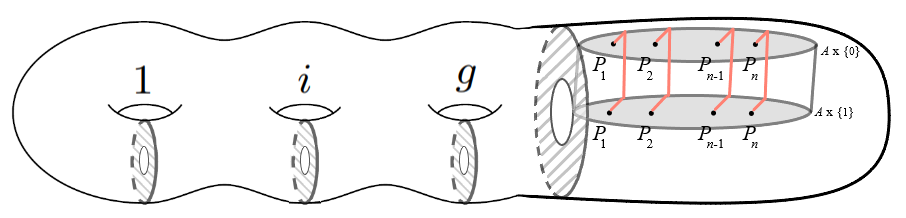}
    \caption{The representation of standard closure in a thickened surface. The red arcs are on the boundary of \(A\times I\).}
    \label{fig:standard_closure_thick}
\end{figure}

A braiding algorithm for unoriented links in thickened surfaces represented as plat closures of surface braids having an even number of moving strands is presented in \cite{cattabriga2018markov}, as well as a plat braid equivalence (analogue of the Markov theorem). The representation of oriented links in thickened surfaces as standard closures of surface braids (as well as their equivalence) is the subject of  \cite{nextone}. We will also prove the braiding here as a corollary, after showing the existence of a transformation from plat to standard closure. 

\subsection{Passing from standard closure to plat closure}

For the passage from standard closure to plat closure of a surface braid, we use the same construction of Section \ref{from_standard}, obtaining the following: 
\begin{theorem}
Let \(\Sigma_g\) be a c.c.o. surface of genus \(g\). Given a link \(L\subset N(\Sigma_g)\), a neighbourhood of \(\Sigma_g\), represented by a braid \(\beta\in \mathcal{B}_{\Sigma_g, m}\) with standard closure, it is algorithmically possible to generate a braid \(B\in \mathcal{B}_{\Sigma_g,2m}\) which represents an isotopic link but with plat closure. The algorithm has a computational complexity of \(O(M)\), where \(M\) is the number of crossings present in the braid \(\beta\).
\end{theorem}

\begin{proof}
The proof builds on Theorems~\ref{alternativeplatbraiding} and \ref{th:stand_plat_h} concerning the passing from standard to plat closure for the cases of $\mathbb{R}^3$ and that of the handlebody. Since the moving subbraid of a surface braid can be considered to have the endpoints on the right part of the fixed subbraid, it is always possible to isotope the closing arcs behind the moving strands without interfering with the \(a_i\) and \(b_i\) loop generators. The algebraic transformation is still equation (\ref{eqn:transformation}). Furthermore, since the substitution is done once for every crossing,  the computational time is linear in the number of the \(\sigma_i\) generators. 
\end{proof}

\subsection{Passing from plat closure to standard closure}

For the transformation of a given surface braid representing a link in a thickened surface via plat closure to an isotopic one but via standard closure, we shall follow the same algorithm as in Section~\ref{Handlebody}, after observing that  Note~\ref{note:cylinder_A} enables us to use the same definition of \(S\)-move as in $S^3$, since in this case the move is done locally in the cylinder \(A \times I\). All this is possible because we represent surface braids via augmented mixed braids in $\mathbb{R}^3$, so we can easily generalize the algorithm to take care of the \(b_i\) loop generators. 

\begin{theorem}\label{th:thick_sur_plat}
Let \(\Sigma_g\) be a c.c.o. surface of genus \(g\). Given a link \(L\subset N(\Sigma_g)\), a neighbourhood of \(\Sigma_g\), represented by a braid \(B\in \mathcal{B}_{\Sigma_g, 2n}\) with plat closure, it is algorithmically possible to generate a braid \(\beta\in \mathcal{B}_{\Sigma_g,m}\) which represents an equivalent link but with standard closure. This algorithm has a computational complexity of \(O(N^2)\) where \(N\) is the number of crossings present in the braid \(B\).
\end{theorem}

\begin{proof}
After giving an orientation to the surface braid \(B\) in plat form, \(a_i\)s and \(b_i\)s generators will possibly be in an upward going strand. 

In the remaining part of the braid we can apply the algorithm of Theorem \ref{th:main}. Since all the \(S\)-moves are done locally in the cylinder \(A \times I\), it is possible to perform the transformation shown in Fig.~\ref{fig:L-C} (and also the one depicted in Fig.~\ref{fig:pass_in_front}), thus turning the strands containing the \(\sigma_i\) generators from plat to standard closure.

Now we deal with the upward \(a_i\)'s and \(b_i\)'s as shown previously; we will twist the piece containing the \(a_i\) or \(b_i\) generator as in Fig.~\ref{fig:twist_a} and Fig.~\ref{fig:twist_b}. In this way, the generators will change their sign, and a type \(d\) crossing will appear into the 0-th column. After performing the \(S\)-move for every twist done in this way, say \(k\) twists in total, we will have: 
\begin{itemize}
    \item any not twisted \(a_i\) becomes \(\sigma_{2k} \sigma_{2k-1} \dots \sigma_{k+1} \sigma_{k}^{-1} \sigma_{k-1}^{-1} \dots \sigma_1^{-1} a_i \sigma_1 \sigma_2 \dots \sigma_{k} \sigma_{k+1}^{-1} \sigma_{k+2}^{-1} \dots \sigma_{2k}^{-1}\)
    \item any twisted \(a_i\) becomes \(a_i^{-1} \sigma_1 \sigma_2 \dots \sigma_{k-1} \sigma_{2k} \sigma_{2k-1} \dots \sigma_{k+1} \sigma_k\)
    \item any not twisted \(b_i\) becomes \(\sigma_{2k} \sigma_{2k-1} \dots \sigma_{k+1} \sigma_{k}^{-1} \sigma_{k-1}^{-1} \dots \sigma_1^{-1} b_i \sigma_1 \sigma_2 \dots \sigma_{k} \sigma_{k+1}^{-1} \sigma_{k+2}^{-1} \dots \sigma_{2k}^{-1}\)
    \item any twisted \(b_i\) becomes \(b_i^{-1} \sigma_1 \sigma_2 \dots \sigma_{k-1} \sigma_{2k} \sigma_{2k-1} \dots \sigma_{k+1} \sigma_k\)
\end{itemize}
In the same way as before, if we have two \(S\)-moves one after the other, without any \(\sigma_1\) in between, we will have to perform a Reidemeister I move in order to bring the braid into standard position, as shown in Fig.~\ref{fig:0_column_B} and Fig.~\ref{fig:0_column_S}. 

Again note that, since all the substitutions made on the loop generators are done in linear time (the \(S\)-moves done in the first column are of the same type), the computational complexity of the algorithm remains quadratic in the number of the \(\sigma_i\) generators.
\end{proof}

\begin{figure}[h!]
    \centering
    \includegraphics[width=.75\textwidth]{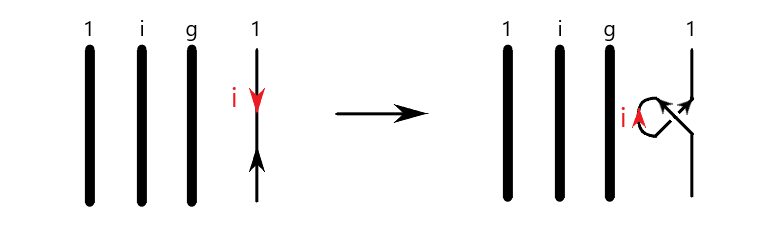}
    \caption{Twisting a \(b_i^k\) belonging to an oriented upward strand changes \(k\) and generates a \(d\)-type crossing. The black arrow marks the orientation of the strand, while the red one marks the direction of \(b_i\) generator.}
    \label{fig:twist_b}
\end{figure}

\begin{figure}[h!]
    \centering
    \includegraphics[width=.31\textwidth]{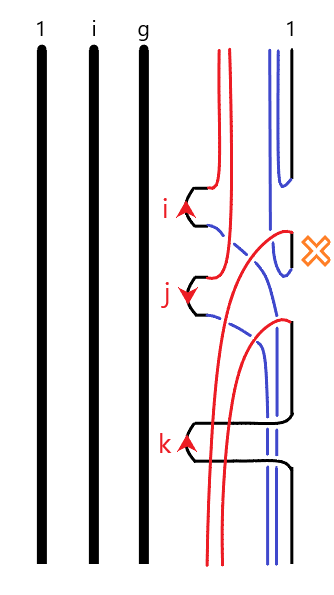}
    \caption{The result after all the twists and \(S\)-moves on the 0-th column over the \(b_i\) generators. The red arrows mark only the direction of the \(b_i\) generators.}
    \label{fig:0_column_B}
\end{figure}

\begin{example}\rm
In Fig.~\ref{fig:example_plat_standard_hand}  is depicted the passing from a braid representing a link via plat closure to a braid representing the same link but via standard closure in the thickened surface setting. First we give an orientation to the link, then we apply all the \(S\)-moves. Thereafter we remove all the extra crossings and twist the loop generators oriented upward. Note that the black arrow represents the sign of the \(b_1\) generator, while the blue one represents its orientation. We apply the last \(S\)-moves in the 0-th column and remove again all the extra crossings obtaining the final braid, which represents the same link class but via standard closure.
\end{example}

\begin{figure}[h!]
    \centering
    \includegraphics[width = \textwidth]{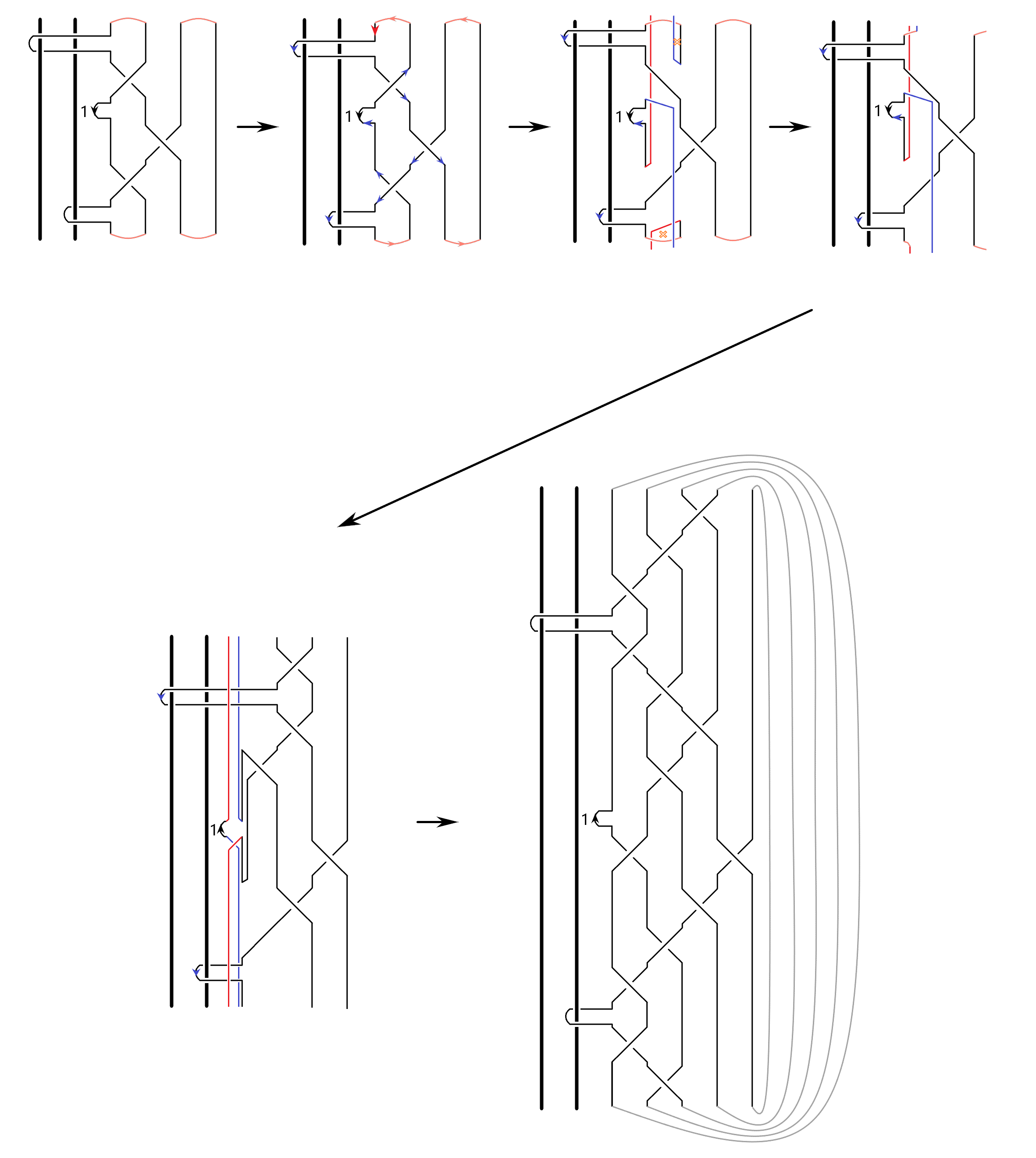}
    \caption{An example of passing from a braid in plat closure to one in standard closure in the thickened surface setting.}
    \label{fig:example_plat_standard_thick}
\end{figure}

As a corollary of  Theorem~\ref{th:thick_sur_plat} and results in  \cite{cattabriga2018markov} (recall the end of Section~\ref{subsec:plat_stand_thick}), we have proven the analogue of the Alexander theorem for an (oriented) link in a thickened surface being isotopic to the standard closure of a surface braid: 
\begin{theorem}
Let \(\Sigma_g\) be a genus \(g\) surface. Every oriented link \(L\) in the thickened surface \(\Sigma_g \times I\) may be braided to a surface braid in \(\mathcal{B}_{\Sigma_g, n}\), the standard closure of which is isotopic to \(L\). 
\end{theorem}

\section{The \text{C++} Algorithm} \label{Algorithm}

In this last section we describe the algorithm implemented in \text{C++} which realizes the procedures described in the previous sections. The source code can be found on \href{https://github.com/Paolo-Cavicchioli/plat_standard_representation}{GitHub} \cite{cavicchioligithub2}. The algorithm takes as input the word associated to a link described by a braid in plat resp. standard closure, and then it computes the word associated to an equivalent link represented by a braid via standard resp. plat closure. 

If the algorithm is asked to compute the plat representation, then it just substitutes each crossing according to Equation~(\ref{eqn:transformation}). 
On the other hand, if the algorithm is asked to compute the standard representation, it gives first an orientation to the plat braid. The pseudo-code implementation is the following: 

\begin{algorithm}
\caption{Giving an orientation to the crossings}\label{alg:cap}
\begin{algorithmic}
\While{$i \neq 1$} \\
    \quad \: First cycle $i \gets 1$ 
    \For{generators \(\in\) word from left} 
         \If{generator is \(\sigma_i\) or \(\sigma_{i-1}\)} 
            \State $i$ follows the strand \\ 
            \quad \: \quad \: \quad \, Add counter to the generator 
        \EndIf 
    \EndFor \\
    \quad \: Change strand once arrived at the bottom closure  
    \For{generators \(\in\) word from right} 
        \If{generator is \(\sigma_i\) or \(\sigma_{i-1}\)} 
            \State $i$ follows the strand 
        \EndIf 
    \EndFor \\
    \quad \: Change strand once arrived at the top closure 
\EndWhile
\end{algorithmic}
\end{algorithm}

The variables are all initialized during the orientation part of the program. 
Before applying the algorithm from Theorem~\ref{th:main}, the program checks if there are any strands oriented upward entirely over or under the rest of the braid, and proceeds to eliminate all the crossings where it is the case. In this way the program is optimized if the braid is given in canonical plat form. 

Then the main part starts. First of all, if we are not in the case of \(\mathbb{R}^3\), we deal with column zero. Every \(a_i\) or \(b_i\) generator is substituted with the relative word, with respect to its orientation (recall Sections~\ref{Handlebody} and~\ref{thickened} for details). Since all the \(S\)-moves are of the same type, we do not need to check the signs of the other \(a_i\)'s and \(b_i\)'s of the same column. For this reason the computational cost of this part of the algorithm is linear in the number of the \(a_i\) and \(b_i\) generators. 

Then, cycling over every column, the number of type \(a, b, c\) and \(d\) crossings (see Fig.~\ref{fig:crossings}) is computed, and their type (over or under crossing) is memorized. Every crossing is then substituted by the appropriate \(C_a, C_b, C_c\) or \(C_d\) formula, checking for each crossing the signs of the other crossings in the same column, in order to use the right signs for the newly generated \(\sigma_i\)'s. For this reason the computational cost of the algorithm is quadratic in the number of the \(\sigma_i\) generators. In this part we also keep track of possible Reidemeister~I move necessities. In this eventuality, we will eliminate the extra crossing, obtaining a well-defined braid. 

Now, for every odd labeled column we deal with the top and bottom part management, following the rules in Section~\ref{top_bottom_management} and checking which one of the cases depicted in Figures~\ref{fig:top_part_1}, \ref{fig:top_part_2}, \ref{fig:top_part_3}, \ref{fig:top_part_4} and \ref{fig:bottom_part_1}, \ref{fig:bottom_part_2}, \ref{fig:bottom_part_3}, \ref{fig:bottom_part_4} occurs. 

Finally, the algorithm gives as output the new word associated to the resulting braid. Since every old generator was substituted in some way, the new word pieces are written following the old sequence of generators. 

The source files of the program can be found in the GitHub repository at \cite{cavicchioligithub2}.

\section{Conclusion}

In this paper we provided algorithms for passing from the plat closure representation of a braid to the standard closure and vice versa, retaining the isotopy class of the resulting link (be it in the classical setting, in a handlebody or in an thickened surface), and estimated the computational complexity.

Our next goals are, first to  construct a direct plat braiding algorithm and a Birman-type theorem concerning plat closures of braids in handlebodies, complementing the results of H\"aring and Lambropoulou for links represented as standard closures of  braids in handlebodies  \cite{haring2002knot}.  Cf.~\cite{handlebody}. Then, to construct a direct braiding algorithm for links represented as standard closures of braids in thickened surfaces and, subsequently, a Markov-type theorem in this setting, complementing the results of Cattabriga and Gabrov{\v{s}}ek for links represented as plat closures of surface braids \cite{cattabriga2018markov}. Cf.~\cite{nextone}. These two gaps in the literature of  knot theory will be the floor to our next steps, namely the generalization of our algorithm to the case of c.c.o. 3-manifolds. 
Let \(M\) be a c.c.o. 3-manifold. We note that a handlebody (or  Heegaard) decomposition   of \(M\) is naturally associated to the plat closure of braids, while a surgery description of \(M\) is naturally associated to the standard closure of braids. 
The two braid equivalence theorems for links in c.c.o. 3-manifolds via the standard closure  \cite{lambropoulou1997markov} and via the plat closure \cite{cattabriga2018markov} are directly related to the description of the 3-manifold. 
The study of the relation between the two  braid equivalence  theorems will be a further step.  

Our results could apply to enhance the computation of link invariants, such as the Jones polynomial, for knots and links in 3-manifolds. 

\bibliographystyle{plain}
\bibliography{main}

\end{document}